\documentclass[a4paper]{article}
\usepackage{latexsym,amsmath,amsthm,amssymb,bbm,fullpage}
\usepackage{graphicx,amsmath,amssymb,latexsym,psfrag}
\usepackage{color}
\usepackage[latin1]{inputenc}

\newcommand{\E}{\ensuremath{\mathbb{E}}} 
\newcommand{\dP}{\ensuremath{\mathbb{P}}} 
\newcommand{\dR}{\ensuremath{\mathbb{R}}} 
\newcommand{\R}{\dR}
\newcommand{\1}{\mathbbm 1}
\newcommand{\cX}{\mathcal{X}}
\newcommand{\cY}{\mathcal{Y}}

\newcommand{\cP}{\mathcal{P}}

\newcommand{\ind}{\mathrm{1}\hskip -3.2pt \mathrm{I}} 
\newcommand{\card}{\mathrm{card}}
\newcommand{\diam}{\mathrm{diam}}

\newtheorem{theorem}{Theorem}
\newtheorem{proposition}[theorem]{Proposition}
\newtheorem{lemma}[theorem]{Lemma}
\newtheorem{corollary}[theorem]{Corollary}

\theoremstyle{definition}

\theoremstyle{remark} \newtheorem{remark}{Remark}

\begin{document}

\title{Combinatorial optimization over two random point sets}

\author{Franck Barthe and Charles Bordenave}
\maketitle

\begin{abstract}
We analyze combinatorial optimization problems over a pair of random point sets $(\cX,\cY)$ in $\R^d$ of equal cardinal. Typical examples include the matching of minimal length, the traveling salesperson tour constrained to alternate between points of each set, or the connected bipartite $r$-regular graph of minimal length. As the cardinal of the sets goes to infinity, we investigate the convergence of such bipartite functionals.
\footnote{\noindent {\em AMS $2000$ Mathematics Subject Classification : } 90C27 (Primary); 39B62 ; 60D05 (Secondary).
}
\end{abstract}

\section{Introduction}
This work pertains to the probabilistic study of Euclidean combinatorial optimization
problems.  The starting point in this field is the celebrated theorem of 
 Beardwood, Halton and
Hammersley \cite{BHH} about the traveling salesperson problem. Its ensures that given
a sequence $(X_i)_{i\ge 1}$ of independent random variables on $\R^d$, $d\ge 2$ with 
common law $\mu$ of bounded support, then almost surely 
$$ \lim_{n\to\infty} n^{\frac1d-1}T(X_1,\ldots, X_n)=\beta_d \int f^{1-\frac1d}.$$
Here $\beta_d$ is a constant depending only on the dimension, $f$ is the density of the 
absolutely continuous part of $\mu$ and 
$$T(X_1,\ldots, X_n)=\inf_{\sigma\in \mathcal S_n} \sum_{i=1}^{n-1} |X_{\sigma(i+1)}-X_{\sigma(i)}|+
  |X_{\sigma(1)}-X_{\sigma(n)}|$$
is the length (for the canonical Euclidean distance) of the shorstest tour through the points
$X_1,\ldots, X_n$. In the above formula $\mathcal S_n$ stands for the set of permutations of
$\{1,2,\ldots,n\}$.
Very informally, this result supports the following interpretation:
  when the number of points $n$ is large, for $\mu$ almost every $x$, if the salesperson is at $X_i=x$
  then  the distance 
 to the next point  in the optimal tour  is 
comparable to $ \beta(d) (nf(x))^{-1/d}$ if $f(x)>0$ and of lower order otherwise. 
This should be compared to the fact that the distance from $X_i=x$ to $\{X_j, j\le n \mbox{ and } j \neq i\}$
also stabilizes at the same rate. 

 Later, Papadimitriou
\cite{P78} and Steele \cite{S81} have initiated a general theory of
Euclidean functionals $F(\{X_1,\ldots,X_n\})$ that satisfy almost sure
limits of this type. We refer the reader to the monographs of Steele
\cite{S} and Yukich \cite{Y} for a full treatment of this now mature
theory, and present a short outline. It is convenient to consider
multisets rather than sets, so throughout the paper 
$\{x_1,\ldots,x_n\}$ will stand for a multiset (the elements are
unordered but may be repeated). The umbrella theorem in \cite{Y} puts forward 
the following three features of a functional $F$ on finite multisets of $\R^d$:
\begin{itemize}
 \item $F$ is $1$-homogeneous if it is translation invariant and dilation covariant:
    $$ F(a+\lambda \cX)=\lambda F(\cX)$$
   for all finite multisets $\cX$, all $a\in \R^d$ and $\lambda\in \R^+$.
 \item  The key assumption is subadditivity: $F$ is subadditive if there exists a constant
  $C>0$ such that for all multisets $\cX,\cY$ in the unit cube $[0,1]^d$, 
   $$ F(\cX\cup\cY)\le F(\cX)+F(\cY)+C.$$
 As noted by Rhee in \cite{rhee1}, this assumption implies that there is another constant $C'$ such that
 for all multiset in $[0,1]^d$,
 \begin{equation}\label{eq:good-bound}
   |F( \cX) | \leq C'  \left( \card( \cX) \right)^{1- \frac{1 }{d} }.  
\end{equation}
Hence  the worst case  for $n$ points is at most in $n^{1-\frac1d}$ and the above mentioned theorems
show that the average case is of the same order.
 \item The third important property is smoothness (or regularity). A functional $F$ on finite multisets 
 $\R^d$  is smooth if there is a constant $C''$ such that for all multisets $\cX,\cY,\mathcal Z$ in 
 $[0,1]^d$, it holds
  $$| F(\cX\cup \mathcal Y)-F(\cX\cup \mathcal Z)|  \le C''   \left( \card( \cY)^{1- \frac{1 }{d}} +
   \card( \mathcal Z)^{1- \frac{1 }{d}} \right).$$
\end{itemize}
These three properties are enough to show upper limits for $F$, on the model of the Beardwood, Halton, Hammersley
theorem. To have the full limits, the umbrella theorem of \cite{Y} also requires to check a few more 
properties of a so-called boundary functional associated with $F$. 

 Next, let us present a classical optimization problem which does not enter the above picture.
 Given two multi-subsets of $\R^d$
with the same cardinality, $\mathcal X=\{X_1,\ldots,X_n\}$ and
$\mathcal Y=\{Y_1,\ldots,Y_n\}$, the cost of the minimal bipartite
matching of $\mathcal X$ and $\mathcal Y$ is defined as
  $$ M_1(\mathcal X,\mathcal Y)=\min_{\sigma\in \mathcal S_n} \sum_{i=1}^{n} |X_i-Y_{\sigma(i)}|,$$
  where the minimum runs over all permutations of $\{1,\ldots,n\}$.
  It is well-known that $n^{-1}
M_1\big(\{X_i\}_{i=1}^n, \{Y_i\}_{i=1}^n\big)$ coincides with the
power of the $L_1$-Wasserstein distance between the empirical distributions
$$  W_1\Big(\frac1n \sum_i \delta_{X_i},\frac1n \sum_i \delta_{Y_i}\Big),$$
hence it is easily seen to tend to $0$, for example when $\mu$ has
bounded support. Recall that given two finite measures $\mu_1$, $\mu_1$ on $\mathbb R^d$
with the same total  mass,
$$ W_1(\mu_1,\mu_2)= \inf_{\pi\in\Pi(\mu_1,\mu_2)}  \int_{\R^d\times \R^d} |x-y| \, d\pi(x,y) ,$$
where $\Pi(\mu_1,\mu_2)$ is the set of measures on $(\mathbb R^d)^2$ having $\mu_1$ as first marginal 
and $\mu_2$ as second marginal (see e.g. \cite{rachev,villani} for more background).
Note that  for all
finite multisets $\cX$, $\cY$ in $[0,1]^d$ with $\card(\cX) =
\card(\cY)$,
$$
M_1( \cX, \cY) \leq \sqrt d \, \card( \cX),
$$
and equality holds for some well-chosen configurations of any cardinal
(all elements in $\cX$ at $(0,\cdots,0)$ and all elements in $\cY$ at
$(1,\cdots,1)$). Hence, an interesting feature of $L$ (as well as
others bipartite Euclidean optimization functionals) is that the
growth bound assumption  \eqref{eq:good-bound} fails, hence it is not subadditive in the above sense.
However Dobri\'c and Yukich have stated the following theorem:
  \begin{theorem}[\cite{DY}] \label{th:main} Let $d\ge 3$ be an integer. Assume that $\mu$ is a probability measure on $\mathbb R^d$ having a bounded support.
 Consider mutually independent random variables $(X_i)_{i\ge 1}$ and $(Y_j)_{j\ge 1}$ having distribution $\mu$.
 Then, almost surely,
$$ \lim_{n}  n^{\frac{1}{d}-1}  M_1\big(\{X_1,\ldots,X_n\}, \{Y_1,\ldots,Y_n\}\big) 
 = \beta_1(d) \int _{\R^d} f^{\frac{d-1}{d}},$$
where $f(x)\, dx$ is the absolutely continuous part of $\mu$ and $\beta_1(d)$ is a constant depending only on the dimension $d$. 
\end{theorem}
When $f$ is not the uniform measure on the unit cube, there is an issue in the proof of \cite{DY} that apparently cannot be easily fixed (the problem lies in their Lemma 4.2 which is used  for proving that the $\liminf$ is at least $ \beta_1(d) \int _{\R^d} f^{\frac{d-1}{d}}$). In any case, the proof of Dobri\'c and Yukich is very specific to the bipartite matching as it uses from 
the start the Kantorovich-Rubinstein 
dual representation of the optimal transportation cost. It is not adapted to 
a general treatment of bipartite functionals. The starting point of our work was recent paper
of Boutet de Monvel and Martin \cite{BM} which (independently of \cite{DY}) establishes the 
convergence of the bipartite matching for uniform variables on the unit cube, without using
the dual formulation of the transportation cost. Building on their approach we are able to 
propose a soft approach of bipartite functionals, based on appropriate notions of subadditivity
and regularity. These properties allow to establish upper estimates on upper limits. In order to 
deal with lower limits we adapt to the bipartite setting the ideas of boundary functionals
exposed in \cite{Y}. We are able to explicitly construct such functionals for a class of optimization
problems involving families of graphs with good properties, and to establish full convergence for 
absolutely continuous laws. Finally we introduce a new notion of inverse subadditivity which allows
to deal with singular parts.

This viewpoint sheds a new light on the result of Dobri\'c and Yukich, that we
extend in other respects, by considering power distance costs, and unbounded random variables 
satisfying certain tail assumptions.  Note that in the classical theory of Euclidean functionals, the analogous question for unbounded random variables was answered
in Rhee \cite{rhee} and generalized in \cite{Y}.

Let us illustrate our results in the case of the bipartite matching with power distance cost : given $p >0$ and two multi-subsets of $\R^d$, $\mathcal X=\{X_1,\ldots,X_n\}$ and
$\mathcal Y=\{Y_1,\ldots,Y_n\}$, define
  $$ M_p(\mathcal X,\mathcal Y)=\min_{\sigma\in \mathcal S_n} \sum_{i=1}^{n} |X_i-Y_{\sigma(i)}|^p,$$
  where the minimum runs over all permutations of $\{1,\ldots,n\}$. Note that we have
the same result for the bipartite travelling salesperson problem, and that our generic approach puts forward
key properties that allow to establish similar facts for other functionals.
As mentioned in the title, our results apply to relatively high dimension. More precisely, if the length  of 
edges are counted to a power $p$, our study applies to dimensions $d>2p$ only.

 \begin{theorem} \label{th:mainMp} Let $0 < 2p < d$. Let $\mu$ be a probability measure on $\mathbb R^d$ with absolutely continuous part $f(x)\, dx$. We assume  that  for some $\alpha >   \frac{4dp}{ d -2p}$,
$$
\int |x|^{\alpha} d \mu (x) < +\infty.
$$
 Consider mutually independent random variables $(X_i)_{i\ge 1}$ and $(Y_j)_{j\ge 1}$ having distribution $\mu$. Then
 there are positive constants    $\beta_p(d),\beta'_p(d) $ depending only on $(p,d)$ such that 
 the following convergence holds almost surely  
\begin{eqnarray*}
 \limsup_{n}  n^{\frac{p}{d}-1}  M_p \big(\{X_1,\ldots,X_n\}, \{Y_1,\ldots,Y_n\}\big) 
& \leq &\beta_p(d) \int _{\R^d} f^{1-\frac{p}{d}}, \\
 \liminf_{n}  n^{\frac{p}{d}-1}  M_p \big(\{X_1,\ldots,X_n\}, \{Y_1,\ldots,Y_n\}\big) 
& \geq & \beta'_p(d) \int _{\R^d} f^{1-\frac{p}{d}}. \nonumber
\end{eqnarray*}
Moreover 
$$
 \lim_{n}  n^{\frac{p}{d}-1}  M_p \big(\{X_1,\ldots,X_n\}, \{Y_1,\ldots,Y_n\}\big) 
=  \beta_p(d) \int _{\R^d} f^{1-\frac{p}{d}}
$$
provided one of the following hypothesis is verified:
\begin{itemize}
	\item $\mu$  is the uniform distribution over a bounded set $\Omega\subset \R^d$ with positive Lebesgue measure. 
	\item $d\in\{1,2\}$, $p\in(0,d/2)$ or $d\ge 3$, $p\in (0,1]$, and $f$ is up to a multiplicative constant the indicator function over a bounded set $\Omega\subset \R^d$  with positive Lebesgue measure. 
\end{itemize}	
\end{theorem}

Our constant $\beta'(d)$ has an explicit expression in terms of the cost of an optimal boundary matching for the uniform measure on $[0,1]^d$ (see Lemma \ref{lem:liminfBM}). We strongly suspect that $\beta_p(d) = \beta'_p(d) $ but we have not been able to solve this important issue. Also, assuming only $\alpha>\frac{2dp}{ d -2p}$, we can establish convergence in probability.
A basic concentration inequality implies that if $\mu$ has bounded support the convergence holds also in $L^q$ for all $q \geq 1$.

The paper is organized as follows: Section~\ref{sec:general} presents the key properties for
 bipartite functionals (homogeneity, subadditivity and regularity) and gathers
useful preliminary statements. Section~\ref{sec:cube} establishes the convergence
for uniform samples on the cube. Section \ref{sec:upper} proves upper bounds on the upper
limits. These two  sections essentially rely on classical subadditive methods, nevertheless
a careful analysis is needed to control the differences of cardinalities of the two 
samples in small domains. 
In Section \ref{sec:ex}, we introduce some examples of bipartite functionals. The lower limits are harder to prove and require a new notion of penalized boundary functionals.
It is however difficult to build an abstract theory there, so in Section~\ref{sec:lower}, we will first present the proof for bipartite
matchings with power distance cost, and put forward a few lemmas which will be useful for other
functionals. We then check that for a natural family of Euclidean combinatorial optimization functionals defined in \S \ref{subsec:combopt}, the lower limit also holds. This family includes the bipartite traveling salesman tour. Finally, Section~\ref{sec:final} mentions possible variants and extensions.

\section{A general setting} \label{sec:general}
Let $\mathcal M_d$ be the set of all finite multisets contained in $\R^d$. We consider a bipartite
functional:
$$L:\mathcal M_d\times \mathcal M_d\to \R^+.$$
Let $p>0$. We shall say that $L$ is $p$-homogeneous if for all multisets $\mathcal X,\mathcal Y$, all $a\in \R^d$
and all $\lambda \ge 0$,
\begin{equation*}\label{eq:Hp} 
 L(a+\lambda \mathcal X, a+\lambda \mathcal Y)=\lambda^p L(\mathcal X,\mathcal Y). \tag{$\mathcal H_p$} 
\end{equation*}
Here $a+\lambda\{x_1,\ldots,x_k\}$ is by definition $\{a+\lambda x_1,\ldots, a+\lambda x_k\}$.
For shortness, we call the above property $(\mathcal H_p)$. Note that a direct consequence is that $L(\emptyset,\emptyset)=0$.

The functional $L$ satisfies the regularity property $(\mathcal R_p)$ if there exists a number $C$ such that for all multisets $\mathcal X,\mathcal Y,
\mathcal X_1,\mathcal Y_1,\mathcal X_2,\mathcal Y_2$, denoting by $\Delta$ the diameter of their union, the following
inequality holds
\begin{equation*}\label{eq:Rp}
L(\mathcal X\cup \cX_1,\cY \cup\cY_1)\le L(\mathcal X\cup \cX_2,\cY \cup\cY_2)+C \Delta^p \big( \mathrm{card}(\cX_1)+
\mathrm{card}(\cX_2)+\mathrm{card}(\cY_1)+\mathrm{card}(\cY_2)\big).\tag{$\mathcal R_p$} 
\end{equation*} 
The above inequality implies in particular an easy size bound: $L(\cX,\cY)\le C \Delta^p (\mathrm{card}(\cX)+
\mathrm{card}(\cY))$ when $L(\emptyset,\emptyset)=0$.

Eventually, $L$ verifies the subbaditivity property $(\mathcal S_p)$ if there exists a number $C$ such that for every 
$k\ge 2$ and all multisets $(\cX_i,\cY_i)_{i=1}^k$, denoting by $\Delta$ the diameter of their union, the following
inequality holds
\begin{equation*}\label{eq:Sp}
L\Big( \bigcup_{i=1}^k \cX_i, \bigcup_{i=1}^k \cY_i\Big) \le \sum_{i=1}^k L(\cX_i,\cY_i)+ C\Delta^p \sum_{i=1}^k 
\Big( 1+\big|\mathrm{card}(\cX_i)-\mathrm{card}(\cY_i)\big|\Big).\tag{$\mathcal S_p$} 
\end{equation*}

\begin{remark}
A less demanding notion of "geometric subadditivity" could be introduced by requiring the above 
inequality only when the multisets $\cX_i\cup\cY_i$ lie in disjoint parallelepipeds (see \cite{Y}
where such a notion is used in order to encompass more complicated single sample functionals). It is 
clear from the proofs that some of our results hold assuming only geometric subadditivity (upper limit
for bounded absolutely continuous laws for example). We will not push this idea further in this paper.  
\end{remark}

We will see later on that suitable extensions of the bipartite matching, of the bipartite traveling salesperson
problem, and of the minimal bipartite spanning tree with bounded maximal degree satisfy all these properties. Our main generic result on  bipartite functionals is the following.

\begin{theorem}
\label{th:mainup}
Let $d>2p>0$ and let $L$ be a bipartite functional on $\R^d$ with the properties $(\mathcal H_p)$, $(\mathcal R_p)$ and $(\mathcal S_p)$. 
Consider  a probability measure $\mu$ on $\mathbb R^d$ such that there exists $\alpha>\frac{4dp}{d-2p}$ with
$$\int |x|^\alpha d\mu(x)<+\infty.$$  Consider mutually independent random variables $(X_i)_{i\ge 1}$ and $(Y_j)_{j\ge 1}$ having distribution $\mu$.
Let $f$  be a density function for the absolutely continuous part of $\mu$, 
then, almost surely, 
   \begin{equation*}
   \limsup_{n\to \infty} \frac{L(\{X_1 , \cdots, X_n\}, \{Y_1, \cdots, Y_n\} )}{n^{1-\frac{p}{d}}}\le \beta_L \int f^{1-\frac{p}{d}} ,
   \end{equation*}   
for some constant $\beta_L$ depending only on $L$.   Moreover, if $\mu$ is the uniform distribution over a bounded set $\Omega$  with positive Lebesgue measure, then there is equality: almost surely,
   \begin{equation*}
   \lim_{n\to \infty} \frac{L(\{X_1 , \cdots, X_n\}, \{Y_1, \cdots, Y_n\} )}{n^{1-\frac{p}{d}}} =  \beta_L \mathrm{Vol}(\Omega)^{\frac p d}.
   \end{equation*}   
\end{theorem}

Beyond uniform distributions, lower limits are harder to obtain. In Section~\ref{sec:lower}, we will state a matching lower bound for a subclass of bipartite functionals which satisfy the properties $(\mathcal H_p)$, $(\mathcal R_p)$ and $(\mathcal S_p)$ (see the forthcoming Theorem \ref{th:Llower} and, for the bipartite traveling salesperson tour, Theorem \ref{th:bTSP}).

\begin{remark}
Let $B(1/2) = \{x \in \dR^d: |x|\leq 1/2 \}$ be the Euclidean ball of radius $1/2$ centered at the origin. It is immediate that the functional $L$ satisfies the regularity property $(\mathcal R_p)$ if it satisfies property $(\mathcal H_p)$ and if  for all multisets $\mathcal X,\mathcal Y,
\mathcal X_1,\mathcal Y_1,\mathcal X_2,\mathcal Y_2$ in $B(1/2)$, 
\begin{equation*}\label{eq:R}
L(\mathcal X\cup \cX_1,\cY \cup\cY_1)\le L(\mathcal X\cup \cX_2,\cY \cup\cY_2)+C \big( \mathrm{card}(\cX_1)+
\mathrm{card}(\cX_2)+\mathrm{card}(\cY_1)+\mathrm{card}(\cY_2)\big). \tag{$\mathcal R$} 
\end{equation*} 
Similarly, $L$ will enjoy the subbaditivity property $(\mathcal S_p)$ if it satisfies property $(\mathcal H_p)$ and if for every 
$k\ge 2$ and all multisets $(\cX_i,\cY_i)_{i=1}^k$ in $B(1/2)$, 
\begin{equation*}\label{eq:S}
L\Big( \bigcup_{i=1}^k \cX_i, \bigcup_{i=1}^k \cY_i\Big) \le \sum_{i=1}^k L(\cX_i,\cY_i)+ C   \sum_{i=1}^k 
\Big( 1+\big|\mathrm{card}(\cX_i)-\mathrm{card}(\cY_i)\big|\Big). \tag{$\mathcal S$} \end{equation*}
The set of assumptions $(\mathcal H_p)$, $(\mathcal R_p)$, $(\mathcal S_p)$ is thus equivalent to the set of assumptions $(\mathcal H_p)$, $(\mathcal R)$, $(\mathcal S)$.  
\end{remark}

\subsection{Consequences of regularity}

\subsubsection{Poissonization}
For technical reasons, it is convenient to consider the poissonized
version of the above problem. 
Let $(X_i)_{i\ge 1}, (Y_i)_{i\ge 1}$ be mutually independent variables with distribution $\mu$.
Considering independent variables $N_1$,
$N_2$ with Poisson distribution $\mathcal P(n)$, the randoms sets
$\{X_1,\ldots, X_{N_1}\}$ and $\{Y_1,\ldots, Y_{N_2}\}$ are
independent Poisson point processes with intensity measures $n\mu$.
For shortness, we set
  $$ L(n\mu):= L\big(\{X_1,\ldots, X_{N_1}\},\{Y_1,\ldots, Y_{N_2}\}\big).$$
 When  $d\mu(x)=f(x)\, dx$ we write $L(nf)$ instead of $L(n\mu)$.
  Note that whenever we are dealing with Poisson processes, $n\in
  (0,+\infty)$ is not necessarily an integer.  More generally $L(\nu)$
  makes sense for any finite measure, as the value of the functional $L$ for 
   two independent Poisson point processes with intensity $\nu$.

Assume for a moment that  the measure $\mu$ has a bounded support, of diameter $\Delta$.
The regularity property ensures that
$$ \left| L(\{X_1,\ldots,X_n\} ,\{Y_1,\ldots,Y_n\} )-L(\{X_1,\ldots,X_{N_1}\} ,\{Y_1,\ldots,Y_{N_2}\} )\right|
\le C \Delta^p \big( |N_1-n|+|N_2-n|\big).$$
Note that  $\E |N_i-n|\le \big(\E(N_i-n)^2\big)^{1/2}= \mathrm{Var}(N_i)=\sqrt n$. Hence the difference between
$  \E L(\{X_i\}_{i=1}^n,\{Y_i\}_{i=1}^n)$ and $\E L(n\mu)$ is at most a constant times $\sqrt n =o(n^{1-p/d})$ when
$d>2p$. Hence in this case, the original quantity and the poissonized version are the same in average at the relevent scale $n^{1-p/d}$. 
The boundedness assumption can actually be relaxed.  To show this, we  need a lemma.

\begin{lemma}
\label{lem:ET}
Let $\alpha > 0$, $n > 0$ and let $\mu $ be a probability measure on $\R^d$ such that for all $t>0$,
$\mu\big(\{x;\; |x| \geq t\}\big)  \le c\, t ^{-\alpha}.$ Let $\cX$, $\cY$ be two independent Poisson point processes of intensity $n \mu$ and $T_n =  \max \{ | Z| : Z \in \cX \cup \cY  \}$.  Then, for all $0 <  \gamma < \alpha$ there
exists a constant $K=K(c,\alpha,\gamma)$ such that for all $n\ge 1$,
$$ \E [T_n^\gamma ]^{\frac1\gamma}\le K n^{ \frac{1}{\alpha}}.
$$ 
Moreover the same conclusion holds if  $\cX = \{X_1,\ldots,X_n\}$, $\cY = \{Y_1,\ldots,Y_n\}$ are two mutually independent sequences of $n$ variables with distribution $\mu$.  
\end{lemma}
\begin{proof}
For $t \geq 0$, let $A_t = \{ x \in \R^d : |x| \geq t \}$ and $g(t) = \int_{ A_t } d\mu$. By assumption,  $\mu(A_t) \leq c t^{-\alpha}$. We start with the Poisson case. 
Since $\cX$, $\cY$ are independent, we have $\dP ( T_n < t) = \dP ( \cX \cap A_t = \emptyset ) ^2=e^{-2n\mu(A_t)}$.
 Therefore, using $1-e^{-u}\le \min(1,u)$,
\begin{eqnarray*}
\E  [T_n^\gamma] & = &   \gamma \int_0 ^\infty t^{\gamma -1} \dP ( T_n \geq  t) dt \\
  &=&  \gamma \int_0 ^\infty t^{\gamma -1}  (1 - e^{- 2 n \mu(A_t) }) dt \\
  & \leq &  \gamma  \int_0 ^{n ^{1/\alpha}} t^{\gamma -1} dt + \int_{n ^{1/\alpha}}^\infty   2 n c t^{\gamma - \alpha-1}  dt \\
  & =&   n ^{\gamma /\alpha}  +    \frac{2c }{\alpha -\gamma}  n^{\gamma/\alpha}, 
\end{eqnarray*}
For the second case, since $\dP ( T_n\ge  t)=1-(1-\mu(A_t))^{2n}\le \min(1, 2n\mu(A_t))$  the same conclusion holds. 
\end{proof}

\begin{proposition}\label{prop:poisson}
Let $d>2p>0$.
Let $\mu$ be a probability measure on $\R^d$ such that $\int |x|^\alpha \, d\mu(x)<+\infty$ for some $\alpha> \frac{2dp}{d-2p}$.
Let $(X_i)_{i\ge 1}, (Y_i)_{i\ge 1}$ be mutually independent variables with distribution $\mu$.
If $L$ satisfies the regularity property $(\mathcal R_p)$ then
$$ \lim_{n\to \infty}\frac{\E L(\{X_i\}_{i=1}^n,\{Y_i\}_{i=1}^n)-\E L(n\mu)}{n^{1-\frac{p}{d}}}=0.$$
\end{proposition}
\begin{remark}
We have not yet proved the finiteness of the above integrals. This will be done later.
The proof below show that the expectations are finite at the same time. So the above statement is 
established with the convention $\infty-\infty=0$.
\end{remark}
\begin{proof}
Let
$N_1$ and $N_2$ be Poisson random variables with mean value $n$. Let $T =  \max \{ | Z| : Z \in \{X_1, \cdots, X_{N_1}\} \cup \{Y_1, \cdots, Y_{N_2}\}  \}$ and $S=  \max \{ | Z| : Z \in \{X_{1} , \cdots, X_{n}\} \cup \{Y_{1}, \cdots, Y_{n}\}  \}$, with the convention that the maximum over an empty set is $0$. The regularity property ensures
 that
$$
\left| L\big(\{X_1,\ldots,X_n\}, \{Y_1,\ldots,Y_n\}\big) -
   L\big(\{X_1,\ldots,X_{N_1}\}, \{Y_1,\ldots,Y_{N_2} \}\big) \right|
\leq C( T +S)^p   \left(|N_1 - n|+|N_2 - n |\right).
$$
Taking expectation gives, using Cauchy-Schwarz inequality and the bound $(a+b)^q \le \max(1,2^{q-1}) (a^q+b^q)$ valid
for $a,b,q>0$
\begin{align*}
& \left|\E  L\big(\{X_1,\ldots,X_n\}, \{Y_1,\ldots,Y_n\}\big) -
   L\big(\{X_1,\ldots,X_{N_1}\}, \{Y_1,\ldots,Y_{N_2} \}\big) \right| \\
& \leq c_p \Big(\E [ T^{2p}]+\E [ S^{2p}]\Big)^{\frac12} \Big(\E[|N_1 - n|^2]+\E[|N_2 - n|^2]\Big)^{\frac12}  \\
& =c_p  \sqrt{2n} \Big(\E [ T^{2p}]+\E [ S^{2p}]\Big)^{\frac12} 
\end{align*}
Since $\alpha > 2p$, by Lemma \ref{lem:ET}, for some $c >0$ and all $n \geq 1$,  $ \E [ T^{2p}]  \leq c n ^{2p/\alpha}$ and $ \E [ S^{2p}]  \leq c n ^{2p/\alpha}$. Hence the above difference of expectations is at most a constant times
$n^{\frac{p}{\alpha}+\frac12}$, which is negligeable with respect to $n^{1-\frac{p}{d}}$ since $\alpha$ is assumed 
to be large enough.
\end{proof}

\subsubsection{Approximations}
\begin{proposition}\label{prop:lipschitz}
Assume that a bipartite functional $L$ satisfies the regularity property $(\mathcal R_p)$.
Let $m,n>0$ and $\mu$ be a probability measure with support included in a set $Q$. Then
$$ \E L(n\mu)\le \E L(m\mu)+C \mathrm{diam}(Q)^p |m-n|.$$
\end{proposition}
\begin{proof}
Assume $n<m$ (the other case is treated in the same way). Let $(X_i)_{i\ge 1}, (Y_i)_{i\ge 1},N_1,N_2,K_1,K_2$
be mutually independent random variables, such that for all $i\ge 1$,  $X_i$ and $Y_i$ have law $\mu$, and for $j\in \{1,2\}$, the law of $N_j$ is $\mathcal P(n)$ and the law of $K_j$ is $\mathcal P(m-n)$.
Then $M_i=N_i+K_i$ is $\mathcal P(m)$-distributed. Then $\{X_1,\ldots,X_{N_1}\}$ and $\{Y_1,\ldots,Y_{N_2}\}$
are independent Poisson point processes of intensity $n\mu$, while  Then $\{X_1,\ldots,X_{M_1}\}$ and $\{Y_1,\ldots,Y_{M_2}\}$
are independent Poisson point processes of intensity $m\mu$. 
By the regularity property,
$$ L\big(\{X_1,\ldots,X_{N_1}\},\{Y_1,\ldots,Y_{N_2}\}\big)\le L\big(\{X_1,\ldots,X_{N_1+K_1}\},\{Y_1,\ldots,Y_{N_2+K_2}\} \big)+ C \mathrm{diam}(Q)^p (K_1+K_2).$$
Taking expectations gives the claim.
 \end{proof}

Applying the above inequality for $m=0$ gives a weak size bound on $\E L(\nu)$.
\begin{corollary}\label{cor:trivialbound}
Assume that $L$ satisfies $(\mathcal R_p)$ and $L(\emptyset,\emptyset)=0$ (a consequence of e.g. $(\mathcal H_p)$),
then if $\nu$ is a finite measure with support included in a set $Q$,
$$ \E L(\nu)\le C \mathrm{diam(Q)}^p \,\nu(Q).$$
\end{corollary}

Recall the total variation distance of two probability measures on $\R^d$ is defined as
$$
d_\mathrm{TV}(\mu,\mu') = \sup \{|\mu ( A) - \mu'(A)|  : A \hbox{ Borel set of $\R^d$}\}. 
$$
\begin{proposition}\label{prop:approx}
Assume that $L$ satisfies $(\mathcal R_p)$. Let $\mu,\mu'$ be two probability measures on $\R^d$ with bounded supports.
Set $\Delta$ be the diameter of the union of their supports. Then 
$$  \E L(n\mu)\le  \E L(n\mu') + 4C \Delta^p\, n \, d_\mathrm{TV}(\mu,\mu').$$
\end{proposition}
\begin{proof}
  The difference of expectations is estimated thanks to a proper
   coupling argument. Let $\pi$ be a probability measure on
   $\R^d\times \R^d$ having $\mu$ as its first marginal and $\mu'$ as
   its second marginal. We consider mutually independent random
   variables $N_1,N_2, (X_i,X_i')_{i\ge 1}, (Y_i,Y_i')_{i\ge 1}$ such
   that $N_1, N_2$ are $\mathcal P(n)$ distributed and for all $i\ge
   1$, $ (X_i,X_i')$ and $(Y_i,Y_i')$ are distributed according to
   $\pi$.  Then the random multisets
 $$ \mathcal X=\{X_1,\ldots,X_{N_1}\} \quad \mathrm{and} \quad  \mathcal Y=\{Y_1,\ldots,Y_{N_2}\}$$
 are independent Poisson point processes with intensity measure
 $n\mu$. Similarly $ \mathcal X'=\{X'_1,\ldots,X'_{N_1}\}$ and $
 \mathcal Y'=\{Y'_1,\ldots,Y'_{N_2}\}$ are independent Poisson point
 processes with intensity measure $n\mu'$. 
  
  The regularity property ensures that 
  \begin{eqnarray*} 
  \lefteqn{ L\big(\{X_1,\ldots,X_{N_1}\},\{Y_1,\ldots,Y_{N_2}\}\big)}\\ 
  &\le& L\big(\{X'_1,\ldots,X'_{N_1}\},\{Y'_1,\ldots,Y'_{N_2}\}\big)+2C\Delta^p \left(\sum_{i=1}^{N_1} \1_{X_i\neq X'_i}+
  \sum_{j=1}^{N_2} \1_{Y_j\neq Y'_j} \right).
 \end{eqnarray*} 
 Taking expectations yields
 \begin{eqnarray*}
 \E L(n\mu)&\le&  \E L(n\mu') +2C\Delta^p \E \left(\sum_{i=1}^{N_1} \mathbb P (X_i\neq X'_i)+
  \sum_{j=1}^{N_2} \mathbb P(Y_j\neq Y'_j) \right) \\
   &=& \E L(n\mu') +4C\Delta^p\, n  \, \pi\big(\{(x,y)\in (\R^d)^2;\; x\neq y\}\big).
 \end{eqnarray*}
 Optimizing the later term on the coupling $\pi$ yields the claimed inequality involving the total variation distance.
\end{proof}

\begin{corollary}\label{cor:couplingaverage}
 Assume that the functional $L$ satisfies the regularity property $(\mathcal R_p)$.
   Let $m>0$, $Q\subset \R^d$ be measurable with positive Lebesgue
   measure and let $f$ be a nonnegative locally integrable function on
   $\R^d$. Let $\alpha=\int_Q f/\mathrm{vol}(Q)$ be the average value
   of $f$ on $Q$. It holds
  $$ \E L(m\,f \1_Q) \le \E L(m\alpha \1_Q) + 2C m\, \mathrm{diam}(Q)^p \, \int_Q|f(x)-\alpha|\, dx.$$
\end{corollary}
\begin{proof}
   We simply apply the total variation bound of the previous lemma
   with $n=m\int_Q f=m\alpha\, \mathrm{vol}(Q)$, $d\mu(x)= f(x)\1_Q(x)
   dx/\int_Q f $ and $d\mu'(x)= \1_Q(x)dx/\mathrm{vol}(Q)$. Note that
 $$ 2d_{TV}(\mu,\mu')=\int \Big| \frac{f(x) \1_Q(x)}{\int_Q f}-\frac{\1_Q(x)}{\mathrm{vol(Q)}}\Big|\, dx=
 \frac{\int_Q|f(x)-\alpha|\, dx}{\int_Q f}\cdot$$
\end{proof}

\subsubsection{Average is enough}
It is known since the works of Rhee and Talagrand that concentration inequalities often allow to deduce almost sure
convergence from convergence in average. This is the case in our general setting.

\begin{proposition}\label{prop:conc}
Let $L$ be a bipartite functional on multisets of $\R^d$, satisfying the regularity property $(\mathcal R_p)$.
Assume  $d>2p>0$. Let $\mu$ be a probability measure $\mu$ on $\R^d$ with
  $\int |x|^\alpha d\mu(x)<+\infty.$
Consider independent variables
$(X_i)_{i\ge 1}$ and $(Y_i)_{i\ge 1}$ with distribution $\mu$.

If $\alpha>2dp/(d-2p)$ then the following convergence holds in probability:
$$\lim_{n\to \infty} \frac{L\big(\{X_i\}_{i=1}^n,\{Y_i\}_{i=1}^n\big)-\E L\big(\{X_i\}_{i=1}^n,\{Y_i\}_{i=1}^n\big)}{n^{1-\frac{p}{d}}}=0.$$
Moreover if $\alpha>4dp/(d-2p)$, the  convergence happens almost surely, and if $\mu$ has bounded support, then it also holds in $L^q$ for any $q \geq 1$. 
\end{proposition}
\begin{proof}
This is a simple consequence of Azuma's concentration inequality. It is convenient to $Z(n)=(X_1,\ldots,X_n,Y_1,\ldots,Y_n)$.
Assume first that the support of $\mu$ is bounded and let $\Delta$ denote its diameter. By the regularity property, modifying one point changes the 
value of the functional by at most a constant:
$$ |L(Z_1,\ldots,Z_{2n})-L(Z_1,\ldots, Z_{i-1}, Z'_i,Z_{i+1},\ldots,Z_{2n})|\le 2C \Delta^p.$$
By conditional integration, we deduce that the following martingale difference:
$$ d_i:=\E\big( L(Z(n))\,|\, Z_1,\ldots,Z_i\big) -\E\big( L(Z(n))\,|\, Z_1,\ldots,Z_{i-1}\big) $$
is also bounded $|d_i|\le 2C \Delta^p$ almost surely.
Recall that Azuma's inequality states that 
$$ \mathbb P \left( \big|\sum_{i=1}^{k}  d_i\big|>t\right) \le 2 e^{-\frac{t^2}{2 \sum_i \|d_i\|_\infty^2}}.$$
Therefore, we obtain that 
\begin{equation}\label{eq:conc}
 \mathbb P \Big( \big|L(\{X_i\}_{i=1}^n,
        \{Y_i\}_{i=1}^n)-\E L(\{X_i\}_{i=1}^n,
        \{Y_i\}_{i=1}^n)\big|>t\Big) \le 2 e^{-\frac{t^2}{16n C^2 \Delta^{2p}}},
 \end{equation}
and there is a number $C'$ (depending on $\Delta$ only) such that 
$$ \mathbb P \left( \frac{\big|L(\{X_i\}_{i=1}^n,
        \{Y_i\}_{i=1}^n)-\E L(\{X_i\}_{i=1}^n,
        \{Y_i\}_{i=1}^n)\big|}{n^{1-\frac{p}{d}}}>t\right) \le 2 e^{-C't^2 n^{1-\frac{2p}{d}}}.$$
When $d>2p$, we may conclude by the Borel-Cantelli lemma.

If $\mu$ is not assumed to be of bounded support, let $S:=\max\{|Z_i|;\; i\le 2n\}$. A conditioning argument
allows to use the above method.
Let $s > 0$ and $B(s)=\{x;\; |x|\le s\}$. Given $\{ S \leq s \}$, the variables $\{X_1, \cdots, X_n\}$ and $\{Y_1, \cdots, Y_n\}$ are mutually independent sequences with distribution $\mu_{| B(s)} / \mu(  B(s))$. 
Hence, applying \eqref{eq:conc} for $\mu_{| B(s)} / \mu(  B(s))$ instead of $\mu$ and $2s$ instead of $\Delta$,
 for any $t >0$,
$$
\dP \left( \left| \frac{L\big(\{X_i\}_{i=1}^n,
        \{Y_i\}_{i=1}^n\big)}{ n^{1-\frac{p}{d}}} - \frac { \E
        L\big(\{X_i\}_{i=1}^n, \{Y_i\}_{i=1}^n\big) }{
        n^{1-\frac{p}{d}}} \right| > t  \Bigm| S \leq s \right) \leq 2 \exp \left(
   -\frac{ n ^{ 1 - \frac{ 2p}{ d}} t ^2 }{c_p s^{2p} } \right).
$$ 
Hence for $\delta>0$ to be chosen later,
\begin{align*}
u_n:&= \dP \left( \left| \frac{L\big(\{X_i\}_{i=1}^n,
        \{Y_i\}_{i=1}^n\big)}{ n^{1-\frac{p}{d}}} - \frac { \E
        L\big(\{X_i\}_{i=1}^n, \{Y_i \}_{i=1}^n\big) }{
        n^{1-\frac{p}{d}}} \right| > t    \right) \\
  & \leq \dP( S > n^{\frac 1 \delta} ) + 2 \exp \left(  -\frac{ n ^{ 1 - \frac {2p}{ d} -   \frac{ 2p}{ \delta}  } t ^2 }{ c_p} \right).
\end{align*}
Since $\dP(S>u)=1-(1-\mu(B(s))^{2n}\le 2n \mu(B(s))\le 2n (\int |x|^\alpha d\mu(x))/u^\alpha$, we get that for 
some constant $c$ and any $\delta >0$, 
  $$ u_n \le c n^{1-\frac \alpha \delta } + 2 \exp \left(  -\frac{ n ^{ 1 - \frac{ 2p}{ d} -    \frac{ 2p}{ \delta}  } t ^2 }{ c_p} \right).$$
Since $\alpha > 2dp/(d-2p)$ we may choose $\delta\in [2 dp / ( d -2p),\alpha]$, which ensures that the latter quantities tend
to zero as $n$ increases. This shows the convergence in probability 
 to $0$ of $$ \frac{L\big(\{X_i\}_{i=1}^n,
  \{Y_i\}_{i=1}^n\big)}{ n^{1-\frac{p}{d}}} - \frac { \E
  L\big(\{X_i\}_{i=1}^n, \{Y_i \}_{i=1}^n\big) }{ n^{1-\frac{p}{d}}}.$$
If $\alpha > 4 dp / ( d -2p)$ we may choose we may choose $\delta\in [2 dp / ( d -2p),\alpha/2]$, which ensures
that $\sum_n u_n<+\infty$.  The Borel-Cantelli lemma yields the almost
sure convergence to $0$. 
\end{proof}

\subsection{Consequences of subadditivity}

 We start with a very general statement, which  is
however not very precise when the measures do not have disjoint
supports.
\begin{proposition}\label{cor:mu1mu2} Let $L$ satisfy $(\mathcal S_p)$.
   Let $\mu_1,\mu_2$ be finite measures on $\R^d$ with
   supports included in a set $Q$. Then
 $$ \E L( \mu_1+\mu_2 ) \le \E L(  \mu_1)+ \E L(  \mu_2)
+2C \mathrm{diam}(Q)^p \left(1+ \sqrt{\mu_1(Q)}+
    \sqrt{\mu_2(Q)}\right).$$
\end{proposition}
\begin{proof} Consider four independent Poisson point processes  $\cX_1,\cY_1,\cX_2,\cY_2 $ such that
 for $i\in \{1,2\}$, the intensity of $\cX_i$ and of $\cY_i$ is $\mu_i$.
   It is classical \cite{K} that the random
   multiset $\cX_1\cup \cX_2$ is a Poisson point process with intensity $\mu_1+\mu_2$. Also,
$\cY_1\cup\cY_2$ is an independent copy of the latter process.
Applying the subadditivity property,
\begin{eqnarray*}
\lefteqn{ L(\cX_1\cup\cX_2,\cY_1\cup \cY_2)} \\
& \le & L(\mathcal X_1,\mathcal Y_1)+  L(\mathcal X_2,\mathcal Y_2)+
C \mathrm{diam}(Q)^p \left(1+|\mathrm{card}(\mathcal X_1)
   -\mathrm{card}(\mathcal Y_1)|+ 1+|\mathrm{card}(\mathcal X_2)
   -\mathrm{card}(\mathcal Y_2)| \right).
\end{eqnarray*}    
 Since $\mathrm{card}(\mathcal X_i)$ and $\mathrm{card}(\mathcal Y_i)$ are independent with 
Poisson law of parameter $\mu_i(Q)$ (the total mass of $\mu_i)$,
$$  \E |\mathrm{card}(\mathcal X_i)-\mathrm{card}(\mathcal Y_i)|\le   \left(\E \big(\mathrm{card}(\mathcal X_i)-\mathrm{card}(\mathcal Y_i)\big)^2\right)^{\frac12}=
\sqrt{2\mathrm{var}(\mathrm{card}(\mathcal X_i))}=\sqrt{2\mu_i(Q)}.$$
Hence, taking expectations in the former estimate leads to the claimed inequality.\end{proof}

Partition techniques are essential in the probabilistic theory of Euclidean functionals. 
The next statement allows to apply them to bipartite functionals. In what follows, given a multiset $\cX$
and a set $P$, we set $\cX(P):=\mathrm{card}(\cX\cap P)$.
If $\mu$ is a measure and $f$ a nonnegative function, we write $f
\cdot \mu$ for the measure having density $f$ with respect to $\mu$.
\begin{proposition} \label{prop:partition}
Assume that the functional $L$ satisfies $(\mathcal S_p)$.
 Consider a finite partition   $Q=\cup_{P\in \mathcal P} P$ of a subset of $\R^d$ and
 let  $\nu$ be a measure on $\R^d$ with
 $\nu(Q)<+\infty$. Then
$$ \E L(\1_Q\cdot \nu) \le \sum_{P\in \mathcal P} \E L(\1_P\cdot \nu)+ 
 3C\mathrm{diam}(Q)^p \sum_{p\in \mathcal P}
\sqrt{\nu(P)}.$$
\end{proposition}
\begin{proof}
   Consider $\mathcal X,\mathcal Y$ two independent Poisson point
   processes with intensity $\nu$. Note that $\mathcal X\cap P$ is a
   Poisson point process with intensity $\1_P\cdot \nu$, hence
   $\mathcal X(P)$ is a Poisson variable with parameter $\nu(P)$.
   We could apply the subadditivity property to $(\mathcal X\cap P)_{P\in\mathcal
     P}$, $(\mathcal Y\cap P)_{P\in \mathcal P}$, which  yields
  $$ L(\mathcal X \cap Q, \mathcal Y\cap Q) \le \sum_{P\in \mathcal P}  L(\mathcal X \cap P, \mathcal Y\cap P)
  + C \mathrm{diam}(Q)^p \sum_{p\in \mathcal P} \big( 1+| \mathcal X(P)-
  \mathcal Y(P)|\big).$$
  Nevertheless, doing this gives a contribution at least $ C \mathrm{diam}(Q)^p$ to 
  cells which do not intersect the multisets $\cX,\cY$.
  To avoid this rough estimate, we consider the cells which meet at least one of the multisets:
  $$\widetilde{\mathcal P}:=\{P\in \mathcal P;\; \cX(P)+\cY(P)\neq 0\}.$$
  We get that 
  \begin{eqnarray*}
    L(\mathcal X \cap Q, \mathcal Y\cap Q) &\le& \sum_{P\in \widetilde{\mathcal P}} 
     L(\mathcal X \cap P, \mathcal Y\cap P)
  + C \mathrm{diam}(Q)^p \sum_{p\in \widetilde{\mathcal P}} \big( 1+| \mathcal X(P)-
  \mathcal Y(P)|\big) \\
  &\le&  \sum_{P\in \mathcal P}  L(\mathcal X \cap P, \mathcal Y\cap P)
  + C \mathrm{diam}(Q)^p \sum_{p\in \mathcal P} \1_{\cX(P)+\cY(P)\neq 0}\big( 1+| \mathcal X(P)-
  \mathcal Y(P)|\big)\\
  &\le&  \sum_{P\in \mathcal P}  L(\mathcal X \cap P, \mathcal Y\cap P)
  + C \mathrm{diam}(Q)^p \sum_{p\in \mathcal P} \big(\1_{\cX(P)+\cY(P)\neq 0} +| \mathcal X(P)-
  \mathcal Y(P)|\big).
  \end{eqnarray*}
  Since $\cX(P)$ and $\cY(P)$ are independent Poisson variables with parameter $\nu(P)$,
  $$ \mathbb P\big(\cX(P)+\cY(P) \neq 0\big)=1-e^{-2\nu(P)} \mbox{ and } 
   \E \big|\mathcal X(P)-\mathcal Y(P)\big|\le \sqrt{2\nu(P)}.$$
   Hence, taking expectation and using the bound $1-e^{-t}\le \min(1,t)\le \sqrt{t}$,
   $$ \E L(\1_Q\cdot \nu) \le \sum_{P\in \mathcal P} \E L(\1_P\cdot \nu)+ 
   2 \sqrt 2 C\mathrm{diam}(Q)^p  \sum_{p\in \mathcal P} \sqrt{\nu(P)}.
   $$  
\end{proof}

The next statement deals with iterated  partitions, which  are very useful in the study 
of combinatorial optimisation problems, see e.g. \cite{S,Y}.
If $\mathcal P$ is a partition,
we set $\mathrm{diam}(\mathcal P)=\max_{P\in \mathcal P} \mathrm {diam}( P)$ (the maximal 
diameter of its cells).
 \begin{corollary}\label{cor:part}
 Assume that the functional $L$ satisfies $(\mathcal S_p)$.
 Let $Q\subset \R^d$ and $\mathcal Q_1, \ldots, \mathcal Q_k$ be a sequence of finer and finer
 finite partitions of $Q$. Let $\nu $ be a measure on $\R^d$ with $\nu(Q)<+\infty$. Then
 $$  \E L(\1_Q\cdot \nu) \le \sum_{q\in \mathcal Q_k} \E L(\1_q\cdot \nu)+
 3C \sum_{i=1}^k \mathrm{diam}(\mathcal Q_{i-1})^p \sum_{q \in \mathcal Q_i}
\sqrt{\nu(q)},$$
where by convention $\mathcal Q_0=\{Q\}$ is the trivial partition.
 \end{corollary}
\begin{proof}
We start with applying Proposition~\ref{prop:partition} to the partition $\mathcal Q_1$ of $Q$:
 $$  \E L(\1_Q\cdot \nu) \le \sum_{q\in \mathcal Q_1} \E L(\1_q\cdot \nu)+
 3C  \mathrm{diam}(\mathcal Q_{0})^p \sum_{q \in \mathcal Q_1}
\sqrt{\nu(q)}.$$
Next for each $q\in \mathcal Q_1$ we apply the proposition again for the partition of $q$ induced 
by $\mathcal Q_2$ and iterate the process $k-2$ times.
\end{proof}

\section{Uniform cube samples}\label{sec:cube}
 We introduce a specific
notation for $n\in(0,+\infty)$,
$$ \bar L(n):=\E L\big( n\1_{[0,1]^d}\big).$$
We point out the following easy consequence of the homogeneity properties of Poisson point processes.

\begin{lemma}\label{lem:homo} If $L$ satisfies the homogeneity property $(\mathcal H_p)$ then
   for all  $a\in \R^d$, $\rho>0$ and  $n>0$
    $$ \E L\big(n\1_{a+[0,\rho]^d}\big)=\rho^p \bar L \big(n \rho^d \big).$$
 \end{lemma}

 The following theorem is obtained by adapting to our abstract setting the line of reasoning of Boutet de Monvel and Martin in the paper \cite{BM} which was devoted to the bipartite matching:
 \begin{theorem}\label{th:cube-poisson}
    Let $d>2p$ be an integer. Let $L$ be a bipartite functional on $\R^d$ satisfying the properties $(\mathcal H_p)$, $(\mathcal R_p)$ and $(\mathcal S_p)$. Then there exists $\beta_L\ge 0$ such
    that
$$\lim_{n\to\infty} \frac{\bar L(n)}{n^{1-\frac{p}{d}}}=\beta_L.$$
\end{theorem}
 \begin{proof}
    Let $m\ge 1$ be an integer. Let $K\in \mathbb N$ such that $2^K \le m< 2^{K+1}$. Set $Q_0=[0,a]^d$ where 
    $a:=2^{K+1}/m>1$. Let $\mathcal Q_0=\{Q_0\}$. We consider a sequence of finer and finer partitions 
    $\mathcal Q_j$, $j\ge 1$ where   $\mathcal Q_j$ is a partition 
    of $Q_0$ into $2^{jd}$ cubes of size $a 2^{-j}$ (throughout the paper, this means that the interior of 
     the cells are open cubes of such size, while their closure is a closed cube of the same size. We do not 
     describe precisely how the points in the boundaries of the cubes are partitioned, since it is not relevent
     for the argument). One often says that $\mathcal Q_j$, $j\ge 1$  is a sequence of dyadic partitions of $Q_0$.
    
      A direct application of Corollary~\ref{cor:part} for the partitions $\mathcal Q_1,\ldots, \mathcal Q_{K+1}$
     and the measure $n\1_{[0,1]^d}(x) \, dx$ gives 
      $$\bar L(n) =  \E L(n\1_{[0,1]^d}) \leq \sum_{q\in \mathcal Q_{K+1}} \E L(n \1_{q\cap[0,1]^d}) + 3C \sum_{j=1}^{K+1}
       \mathrm{diam}(\mathcal Q_{j-1})^p \sum_{q\in \mathcal Q_j} \sqrt{n \,\mathrm{Vol}(q\cap [0,1]^d)}.$$
    Note that $\mathcal Q_{K+1}$ is a partition into cubes of size $1/m$, so that its intersection
      with $[0,1]^d$ induces an (essential) partition of the unit cube into $m^d$ cubes of side-length $1/m$.
      Hence, in the first sum, there are $m^d$ terms which are equal, thanks to translation invariance and Lemma~\ref{lem:homo}
       to $\E L(n\1_{[0,m^{-1}]^d})=m^{-p} \bar L( nm^{-d})$. The remaining  terms of the first sum vanish. 
       In order to deal with the second
       sum of the above estimate, we simply use the fact that $\mathcal Q_j$ contains $2^{jd}$ cubical cells
       of size $a 2^{-j}=2^{K+1-j}/m \le 2^{1-j}$. Hence their indidual volumes are at most $2^{d(1-j)}$.
       These observations allow to rewrite the above estimate as
       \begin{eqnarray*}
         \bar L(n)&\le & m^{d-p} \bar L(nm^{-d}) +3C \sum_{j=1}^{K+1}
       \mathrm{diam}( [0,2^{2-j}]^d)^p  2^{jd} \sqrt{n \, 2^{d(1-j)}}\\
         &=&  m^{d-p} \bar L(nm^{-d}) +  3C \sqrt n \,\mathrm{diam}( [0,1]^d)^p \sum_{j=1}^{K+1}
       2^{p(2-j)+\frac{d}{2}(j+1)}.    
       \end{eqnarray*}
         Hence, there is a number  $D$ depending only on $p,d$ and $C$ such that 
         $$ \bar L(n) \le   m^{d-p} \bar L(nm^{-d}) + D \sqrt n \,2^{K (\frac{d}{2}-p)} \le  m^{d-p} \bar L(nm^{-d}) + D\sqrt n \, m^{\frac{d}{2}-p}.$$
         Let $t>0$. Setting, $n=m^d t^d$ and $f(u)=\bar L(u^d)/u^{d-p}$, the latter inequality reads as 
         $$f(mt)\le f(t)+ D t^{p-\frac{d}{2}},$$
         and is valid for all $t>0$ and $m\in \mathbb N^*$. Since $f$ is continuous (Proposition~\ref{prop:lipschitz} 
         shows that $u\mapsto \bar L(u)$ is Lipschitz) and $\lim_{t\to +\infty} t^{p-\frac{d}{2}}=0$, it follows that  $\lim_{t\to +\infty} f(t)$ exists 
         (we refer to \cite{BM} for details). 
 \end{proof}
 
 \begin{remark}
    The above constant $\beta_L$ is positive as soon as 
 $L$ satisfies the following natural condition: for all $x_1,\ldots,x_n,y_1,\ldots y_n$ in $\R^d$, $L(\{x_1,\ldots,x_n\},\{y_1,\ldots,y_n\}) \ge c \sum_i
  \mathrm{dist}(x_i,\{y_1,\ldots,y_n\})^p$.
To see this, one combines Proposition~\ref{prop:poisson} and the lower estimate given in \cite{T}.
 \end{remark}

\section{Upper bounds, upper limits} \label{sec:upper}
\subsection{A general upper bound}
\begin{lemma}
\label{lem:upperL} Let $d>2p$ and let $L$ be a bipartite functional satisfying $(\mathcal S_p)$, 
$(\mathcal R_p)$ and $L(\emptyset,\emptyset)=0$.
Then there exists a constant $D$ such that,  for all finite measures $\nu$, 
$$
 \E L ( \nu) \leq D \, \diam(Q)^p  \min\big( \nu(Q), \nu(Q) ^{1-\frac{p}{d}}\big),$$
where $Q$ contains the support of $\nu$. 
\end{lemma}
\begin{proof}
Thanks to corollary~\ref{cor:trivialbound}, it is enough to deal with the case $\nu(Q)\ge 2^d$ (or any other
positive number).
First note that we may assume that $Q$ is a cube (given a set of diameter $\Delta$, one can find
a cube containing it, with diameter no more than $c$ times $\Delta$ where $c$ only depends on the
norm).
We consider a sequence of dyadic partitions of $Q$, $(\mathcal P_\ell)_{\ell\ge 0}$, where
 for $\ell \in\mathbb N$, 
  $\mathcal P_\ell$ divides $Q$  into $2^{\ell d}$ cubes of
   side-length $2^{-\ell}$ times the one of $Q$. Let $k\in \mathbb N^*$ to be chosen later.
 By Corollary~\ref{cor:part},  we have the following estimate
  \begin{equation}\label{eq:upper20} \E L( \nu )\le \sum_{P\in \mathcal
       P_k} \E L(\1_P\cdot \nu)+
     3C \sum_{\ell=1}^k \big(2^{-\ell+1}\mathrm{diam}(Q)\big)^p \sum_{P\in\mathcal P_\ell} \sqrt{\nu  (P) }.
  \end{equation}
Thanks to Corollary~\ref{cor:trivialbound}, 
the first term of the right-hand side of \eqref{eq:upper20} is at most
$$
\sum_{P\in \mathcal P_k} C\, \big(2^{-k}\mathrm{diam}(Q)\big)^p   \nu (P)
 =C\, 2^{-kp} \big(\mathrm{diam}(Q)\big)^p
   \nu (Q) . 
$$
  By the Cauchy-Schwarz inequality
  $$\sum_{P\in\mathcal P_\ell} \sqrt{ \nu  (P)} \le \left(2^{\ell d} \right)^{\frac12}
  \left( \sum_{P\in\mathcal P_\ell}  \nu  (P)  \right)^{\frac12}= 2^{\frac{\ell d}{2}} \sqrt{\nu(Q)}.$$
 Hence the  second term of the right-hand side of \eqref{eq:upper20} is at most
$$ 3C \big(2\mathrm{diam}(Q)\big)^p  \sum_{\ell=1}^k 2^{\ell \big(\frac{d}{2}-p\big)}  \sqrt{\nu(Q)}
 \le C'   2^{k\big(\frac{d}{2}-p\big)}  \big(\mathrm{diam}(Q)\big)^p \sqrt{\nu(Q)}.$$ 
This leads to 
$$ \E L(\nu)\le  \big(\mathrm{diam}(Q)\big)^p \Big( C  2^{-kp}  \nu(Q) + C'   2^{k\big(\frac{d}{2}-p\big)} 
\sqrt{\nu(Q)}\Big).$$
Choosing $ k = \big\lfloor \frac 1 d \log_2 \nu(Q)  \big\rfloor\ge 1$ completes the proof.
\end{proof}

\subsection{The upper limit for densities}

\begin{theorem}\label{th:upper}
   Let $d>2p$. Let $L$ be a bipartite functional on $\R^d$ satisfying the properties $(\mathcal H_p)$,
   $(\mathcal R_p)$, $(\mathcal S_p)$. Let $f:\R^d\to \R^+$ be an integrable function with
   bounded support. Then
   $$ \limsup_{n\to \infty } \frac{\E L(n\, f)}{n^{1-\frac{p}{d}}} \le \beta_L \int_{\R^d} f^{1-\frac{p}{d}},$$
   where $\beta_L$ is the constant appearing in
   Theorem~\ref{th:cube-poisson}.
\end{theorem}
\begin{proof}
   By a scaling argument, we may assume that the support of $f$ is
   included in $[0,1]^d$ and $\int f =1$ (the case $\int f =0$ is
   trivial). We consider a sequence of dyadic partitions $(\mathcal P_\ell)_{\ell\in\mathbb N}$ of 
   $[0,1]^d$: for $\ell \in\mathbb N$,  $\mathcal P_\ell$ divides  $[0,1]^d$ into $2^{\ell d}$ cubes of
   side-length $2^{-\ell}$. Let $k\in \mathbb N^*$ to be chosen later.
   Corollary~\ref{cor:part} gives
  \begin{equation}\label{eq:upper} \E L(n\, f)\le \sum_{P\in \mathcal
       P_k} \E L(n\, f\1_P)+
     3C\sum_{\ell=1}^k\big( 2^{-\ell+1} \mathrm{diam}([0,1]^d)\big)^p
    \sum_{P\in\mathcal P_\ell} \sqrt{n\int_P f}.
  \end{equation}
  By the Cauchy-Schwarz inequality
  $$\sum_{P\in\mathcal P_\ell} \sqrt{\int_P f} \le \left(2^{\ell d} \right)^{\frac12}
  \left( \sum_{P\in\mathcal P_\ell} \int_P f \right)^{\frac12}=
  2^{\frac{\ell d}{2}} \left(\int f\right)^ \frac12=  2^{\frac{\ell d}{2}}.$$
 Hence the
  second term of the right-hand side of \eqref{eq:upper} is at most
$$ 3C\big( 2 \mathrm{diam}([0,1]^d)\big)^p \sqrt{n} \sum_{\ell=1}^k 2^{\ell \big(\frac{d}{2}-p\big)}
 \le c_d  n^{\frac12} 2^{k\big(\frac{d}{2}-p\big)}.$$
 Let $\alpha_P$ be the average of $f$ on $P$, then applying
Corollary~\ref{cor:couplingaverage} to the first terms of
\eqref{eq:upper} leads to
$$ \E L(n\, f)\le \sum_{P\in \mathcal P_k} \left( \E L(n\, \alpha_P \1_P)+2C\, n\, \mathrm{diam}(P)^p \int_P|f-\alpha_P| \right) +c_d  n^{\frac12} 2^{k\big(\frac{d}{2}-p\big)}.$$
 Each $P$ in the sum is a square of side length $2^{-k}$,
hence using homogeneity (see Lemma~\ref{lem:homo})
\begin{equation}\label{eq:upper2}
   \E L(n\, f)\le \sum_{P\in \mathcal P_k} \left(2^{-kp} M\big(n\, \alpha_P 2^{-kd}\big)+n\,c'_d\, 2^{-kp} \int_P|f-\alpha_P| \right) +c_d  n^{\frac12} 2^{k\big(\frac{d}{2}-p\big)}.
\end{equation}
Let us recast this inequality with more convenient notation. We set
$g(t)=\bar L (t)/t^{1-p/d}$ and we define the piecewise constant function
   $$ f_k=\sum_{P\in \mathcal P_k} \alpha_P \1_P
= \sum_{P\in \mathcal P_k} \frac{\int_P f(x)\, dx }{\mathrm{Vol}(P)} \1_P.$$
   It is plain that $\int f_k=\int f<+\infty$. Moreover, by Lebesgue's
   theorem, $\lim_{k\to \infty} f_k=f$ holds for almost every point
   $x$. Inequality \eqref{eq:upper2} amounts to
   \begin{eqnarray*}
      \frac{\E L(n\, f)}{n^{1-\frac{p}{d}}} &\le &
 \sum_{P\in \mathcal P_k} \left(g\big(n\,\alpha_P 2^{-kd}\big) \alpha_P^{1-\frac{p}{d}} 2^{-kd}+n^{\frac{p}{d}}\,
 c'_d\, 2^{-kp} \int_P|f-f_k| \right) +c_d  n^{\frac{p}{d}-\frac12} 2^{k\big(\frac{d}{2}-p\big)}\\
      &=&\sum_{P\in \mathcal P_k} \left( \int_P g\big(n\, f_k 2^{-kd}\big) f_k^{1-\frac{p}{d}}
 +n^{\frac{p}{d}}\, c'_d\, 2^{-kp} \int_P|f-f_k| \right) +c_d  n^{\frac{p}{d}-\frac12} 2^{k\big(\frac{d}{2}-p\big)}\\
      &=& \int  g\big(n\,2^{-kd} f_k \big) f_k^{1-\frac{p}{d}} + c'_d\,n^{\frac{p}{d}}\,  2^{-kp} \int |f-f_k|
      +c_d \big(n^{\frac1d}2^{-k}\big)^{p-\frac{d}{2}}.
   \end{eqnarray*}
   If there exists $k_0$ such that $f=f_{k_0}$ then we easily get the
   claim by setting $k=k_0$ and letting $n$ go to infinity (since $g$
   is bounded and converges to $\beta_L$ at infinity, see
   Lemma~\ref{lem:upperL} and Theorem~\ref{th:cube-poisson}). On the
   other hand, if $f_k$ never coincides almost surely with $f$, we use
   a sequence of numbers $k(n)\in \mathbb N$ such that
   \begin{equation}\label{eq:kn}
      \lim_n k(n)=+\infty,\quad \lim_n n^{\frac1d} 2^{-k(n)}=+\infty \quad\mbox{and}\quad 
      \lim_n n^{\frac1d} 2^{-k(n)}\left( \int |f-f_{k(n)}|\right)^{\frac{1}{p}}=0.
   \end{equation}
   Assuming its existence, the claim follows easily: applying the
   inequality for $k=k(n)$ and taking upper limits gives
$$    \limsup_{n } \frac{\E L(n\, f)}{n^{1-\frac{p}{d}}} \le \limsup_n \int  g\big(n\,2^{-k(n)d} f_{k(n)} \big) f_{k(n)}^{1-\frac{p}{d}} .$$
Since $\lim f_{k(n)}=f$ a.e., it is easy to see that the limit of the
latter integral is $\beta_L\int f^{1-\frac{p}{d}}$: first the
integrand converges almost everywhere to $\beta_L f^{1-\frac{p}{d}}$
(if $f(x)=0$ this follows from the boundedness of $g$; if $f(x)\neq 0$
then the argument of $g$ is going to infinity). Secondly, the sequence
of integrands is supported on the unit cube and is uniformly
integrable since
 $$\int  \left(g\big(n\,2^{-k(n)d} f_{k(n)} \big) f_k^{1-\frac{p}{d}} \right)^{\frac{d}{d-p}}
 \le (\sup g)^{\frac{d}{d-p}} \int f_{k(n)}= (\sup g)^{\frac{d}{d-p}}
 \int f <+\infty.$$
    
 It remains to establish the existence of a sequence of integers
 $(k(n))_n$ satisfying \eqref{eq:kn}.  Note that since $f_k\ge 0$,
 $\int f_k=\int f=1$ and a.e. $\lim f_k=f$, it follows from
 Scheff\'e's lemma that $\lim_k \int |f-f_k|=0$.  Hence
 $\varphi(k)=(\sup_{j\ge k} \int |f-f_j|)^{-d/p}$ is non-decreasing with
 an infinite limit.  We derive the existence of a sequence with the
 following stronger properties
 \begin{equation}\label{eq:kn2}
    \lim_n k(n)=+\infty,\quad \lim_n \frac{n }{ (2^d)^{ k(n)}}=+\infty \quad\mbox{and}\quad 
    \lim_n   \frac{ n}{(2^d )^{k(n)}\varphi (k(n))}=0
 \end{equation}
 as follows. Set $\gamma=2^d$. Since $\gamma^k \sqrt{\varphi(k-1)}$ is
 increasing with infinite limit
  $$ [\gamma\sqrt{\varphi(0)},+\infty)
=\cup_{k\ge 1} \big[\gamma^k\sqrt{\varphi(k-1)},\gamma^{k+1}\sqrt{\varphi(k)}\big).$$
  For $n\ge \gamma\sqrt{\varphi(0)}$, we define $k(n)$ as the integer such
  that
   $$ \gamma^{k(n)}\sqrt{\varphi(k(n)-1)}\le n < \gamma^{k(n)+1}\sqrt{\varphi(k(n))}.$$
   This defines a non-decreasing sequence.  It is clear from the above
   strict inequality that $\lim_n k(n)=+\infty$. Hence $n \gamma^{-k(n)}\ge
   \sqrt{\varphi(k(n)-1)}$ tends to infinity at infinity. Eventually
   $n/(\gamma^{k(n)}\varphi(k(n))\le \gamma/\sqrt{\varphi(k(n))}$ tends to zero
   as required. The proof is therefore complete.
\end{proof}

\subsection{Purely singular measures}

\begin{lemma}\label{le:sing} Let $d> 2p$. Let $L$ be a bipartite functional on $\R^d$ 
   with properties  $(\mathcal R_p)$ and $(\mathcal S_p)$.
Let $\mu$ be a finite  singular  measure on $\mathbb R^d$ having a  bounded support.
Then
$$\lim_{n\to \infty} \frac{\E L(n\mu)}{n^{1-\frac{p}{d}}}=0 . $$   
\end{lemma}
\begin{proof}
Let $Q$ be a cube which contains the support of $\mu$.
We consider a sequence of dyadic partitions of $Q$, $(\mathcal P_\ell)_{\ell\in\mathbb N}$.
For $\ell\in\mathbb N$, $\mathcal P_{\ell}$ divides $Q$ into 
  $2^{\ell d}$ cubes of  side length $2^{-\ell}$ times the one of $Q$.
As in the proof of Lemma~\ref{lem:upperL}, a direct application of Corollary~\ref{cor:part}
gives for  $k\in \mathbb N^*$:
\begin{equation}\label{eq:upper3}
  \E L(n\mu)\le \sum_{P\in \mathcal
       P_k} \E L(n\1_P\cdot \mu)+
     3C \sum_{\ell=1}^k \big(2^{-\ell+1} \mathrm{diam}(Q)\big)^p \sum_{P\in\mathcal P_\ell} \sqrt{n\mu(P)}.
\end{equation}
The terms of the first sum are estimated again thanks to the easy bound of Corollary~\ref{cor:trivialbound}:
since each $P$ in $\mathcal P_k$  is a cube of side length $2^{-k}$ times the one of $Q$, it holds
 $$  \sum_{P\in \mathcal
       P_k} \E L(n\1_P\cdot \mu) \le  \sum_{P\in \mathcal
       P_k} C \big(2^{-k} \mathrm{diam}(Q) \big)^p n\mu(P) =c_{p,Q}\,  2^{-kp}  n|\mu|.$$
Here $|\mu|$ is the total mass of $\mu$. We rewrite the second term in \eqref{eq:upper3} in terms 
of the function
$$ g_\ell=\sum_{P\in \mathcal P_\ell} \frac{\mu(P)}{\lambda(P)} \1_P
,$$
where $\lambda$ stands for  Lebesgue's measure. Since $\lambda(P)=2^{-\ell d} \lambda(Q)$, we get that 
\begin{eqnarray*}
    \E L(n\mu)&\le&c_{p,Q}\,  2^{-kp}  n|\mu|   + 
     3C\, \big(2 \mathrm{diam}(Q)\big)^p \sqrt{n}\,\sum_{\ell=1}^k 2^{-\ell p} 
  \sum_{P\in\mathcal P_\ell}  2^{\frac{\ell d}{2}} \lambda(Q)^{-\frac12} \lambda(P)\sqrt{\frac{\mu(P)}{\lambda(P)}} \\
    &=& c_{p,Q}\,  2^{-kp}  n|\mu|  +   3C\, \big(2 \mathrm{diam}(Q)\big)^p   \lambda(Q)^{-\frac12}
    \sqrt{n}\,\sum_{\ell=1}^k 2^{\ell\big(\frac{d}{2}- p\big)} 
    \int\sqrt{g_\ell}  .
\end{eqnarray*}
By the differentiability theorem, for Lebesgue-almost every $x$, $g_\ell(x)$ tends to zero when $\ell$ tends
to infinity (since $\mu$ is singular with respect to Lebesgue's measure). Moreover, $g_\ell$ is supported on 
the unit cube and $\int (\sqrt{g_\ell})^2=\int g_\ell=|\mu|<+\infty$. Hence the sequence of functions 
$\sqrt{g_\ell}$ is uniformly integrable and we can conclude that $\lim_{\ell \to \infty}\int \sqrt{g_\ell}=0$.
By Cesaro's theorem, the sequence 
 $$\varepsilon_k=\frac{ \sum_{\ell=1}^{k} 2^{\ell\left( \frac{d}{2}-p\right)} \int \sqrt{g_\ell}}
{ \sum_{\ell=1}^{k} 2^{\ell\left( \frac{d}{2}-p\right)}}$$
also converges to zero, using here that $d> 2p$. By an obvious upper bound of the latter denominator, 
we obtain that there exists a number $c$ which does not depend on $(k,n)$ (but depends on $C,p,d,Q,|\mu|$)
 such that for all $k\ge 1$
$$  \E L(n\mu) \le c\left( n 2^{-kp}  + \sqrt{n }\, 2^{k\left(\frac{d}{2}-p\right)} \varepsilon_k\right),$$
where $\varepsilon_k\ge 0$ and $\lim_k \varepsilon_k=0$. We may also assume that $(\varepsilon_k)$ is non-increasing (the inequality remains valid if one replaces $\varepsilon_k$ by $\sup_{j\ge k} \varepsilon_j$).
It remains to choose $k$ in terms of $n$ in a proper way. Define
$$ \varphi(n)=\sqrt{\varepsilon_{\lfloor \frac1d \log_2 n \rfloor}} ^{\frac{-1}{\frac{d}{2}-p}}.$$
Obviously $\lim_n \varphi(n)=+\infty$. For $n$ large enough, define  $k(n)\ge 1$ as the unique integer
such that 
 $$ 2^{k(n)} \le n^{\frac1d}\varphi(n) < 2^{k(n)+1}.$$
Setting $k=k(n)$, our estimate on the cost of the optimal matching yields
$$ \frac{\E L(n\mu)}{n^{1-\frac{p}{d}}} \le c(d) \left(\frac{2 }{\varphi(n)^p}+\varepsilon_{k(n)} 
  \varphi(n)^{\frac{d}{2}-p} \right) .$$
It is easy to check that the right hand side tends to zero as $n$ tends to infinity. Indeed,
$\lim_n \varphi(n)=+\infty$, hence for $n$ large enough
$$ k(n) \ge \left\lfloor \log_2\left(n^{\frac1d} \varphi(n)/2 \right)\right\rfloor \ge 
  \left\lfloor \frac1d \log_2 n\right\rfloor.$$
Since the sequence $(\varepsilon_k)$ is non-increasing, it follows that 
$$ \varepsilon_{k(n)}   \varphi(n)^{\frac{d}{2}-p} \le  \varepsilon_{  \left\lfloor \frac1d \log_2 n\right\rfloor}   \varphi(n)^{\frac{d}{2}-p} = \sqrt{ \varepsilon_{  \left\lfloor \frac1d \log_2 n\right\rfloor} }$$
tends to zero when $n\to \infty$. The proof is therefore complete.
\end{proof}

\subsection{General upper limits}

The first statement of Theorem \ref{th:mainup} is a consequence of Propositions \ref{prop:poisson}, \ref{prop:conc}, and the following result. 

\begin{theorem}\label{th:up-poisson-general} Let $d>2p>0$. Let $L$ be a bipartite functional
on $\R^d$ with the properties $(\mathcal H_p)$, $(\mathcal R_p)$ and $(\mathcal S_p)$. 
Consider  a finite  measure $\mu$ on $\mathbb R^d$ such that there exists $\alpha>\frac{2dp}{d-2p}$ with
$$\int |x|^\alpha d\mu(x)<+\infty.$$ Let $f$ 
be a density function for the absolutely continuous part of $\mu$,
then
   \begin{equation}\label{eq:glupper}
   \limsup_{n\to \infty} \frac{\E L(n\mu)}{n^{1-\frac{p}{d}}}\le \beta_L \int f^{1-\frac{p}{d}} \cdot 
   \end{equation}   
\end{theorem}
\begin{remark}
Observe that the hypotheses ensure the finiteness of $ \int f^{1-\frac{p}{d}}$. Indeed H\"older's inequality
gives 
$$ \int_{\R^d} f^{1-\frac{p}{d}} \le \left(\int_{\R^d} (1+|x|^\alpha)f(x) dx \right)^{1-\frac{p}{d}} 
\left(\int_{\R^d} (1+|x|^\alpha)^{1-\frac{d}{p}} \right)^{\frac{p}{d}}$$
where the latter integral converges since $\alpha > \frac{2dp}{d-2p}>\frac{dp}{d-p}.$
\end{remark}
\begin{proof} 
Assume first that $\mu$ has a bounded support. Write $\mu=\mu_{ac}+\mu_s$ where $\mu_s$ is the singular part and $d\mu_{ac}(x)=f(x)\, dx$.
Applying Proposition~\ref{cor:mu1mu2} to $\mu_{ac}$ and $\mu_s$, dividing by $n^{1-p/d}$, passing to the
limit and using Theorem~\ref{th:upper} and Lemma~\ref{le:sing} gives
$$ \limsup_n \frac{\E L(n\mu)}{n^{1-\frac{p}{d}}} \le  \limsup_n \frac{\E L(n\mu_{ac})}{n^{1-\frac{p}{d}}} 
  +  \limsup_n \frac{\E L(n\mu_s)}{n^{1-\frac{p}{d}}} \le \beta_L \int f^{1-\frac{p}{d}}.$$
  Hence the theorem is established for measures with bounded supports.
  
  Now, let us consider the general case.
Let $B (t) = \{x \in \R^d : |x | \leq t\}$. Let $A_0 = B (2)$ and for integer $\ell \geq 1$, $A_\ell = B (2^{\ell+1})\backslash B (2^{\ell})$.  Now, let $\cX = \{ X_1, \cdots, X_{N_1}\}$,  $\cY = \{ Y_1, \cdots, Y_{N_2}\}$ be two independent Poisson process of intensity $n\mu$, and $T =  \max \{ | Z| : Z \in \cX \cup \cY  \}$.
Applying the subadditivity property like in the proof of Proposition~\ref{prop:partition}, we obtain 
\begin{eqnarray}
L ( \cX , \cY ) & \leq & \sum_{\ell \geq  0}   L ( \cX \cap A_\ell ,  \cY \cap A_\ell ) + C T^p   \sum_{\ell \geq  0}     \1_{\cX  (A_\ell) +\cY( A_\ell)\neq 0}\big(1+|\cX  (A_\ell) -\cY( A_\ell)|\big) .   \label{eq:UP1} 
\end{eqnarray}
Note that the above sums have only finitely many non-zero terms, since $\mu$ is finite.
We first deal with the first sum in the above inequality. By Fubini's Theorem,
$$
\E   \sum_{\ell \geq 0 }  \frac{L ( \cX \cap A_\ell ,  \cY \cap A_\ell )  }{n^{1-\frac{p}{d}}} =  \sum_{\ell \geq 0}  \E \frac{L ( \cX \cap A_\ell ,  \cY \cap A_\ell )  }{n^{1-\frac{p}{d}}}.
$$
Applying  \eqref{eq:glupper} to the compactly supported measure $\mu_{|A_\ell}$ for every integer $\ell$ gives  
\begin{equation}\label{eq:limsupAl}
\limsup_n \E \frac{L ( \cX \cap A_\ell ,  \cY \cap A_\ell )  }{n^{1-\frac{p}{d}}} \le   \beta_L \int_{A_l}  f^{1-\frac{p}{d}}. 
\end{equation}
By  Lemma \ref{lem:upperL}, for some constant $c_d$, 
$$
\E \frac{L ( \cX \cap A_\ell ,  \cY \cap A_\ell )  }{n^{1-\frac{p}{d}}} \leq c_d 2^{\ell p} \mu (A_\ell)^{1-\frac{p}{d}}. 
$$ 
From Markov inequality, with $m_\alpha = \int |x|^\alpha d \mu (x)$,
$$
\mu (A_\ell) \leq \mu ( \R^d \backslash B (2^{\ell})) \leq  2^{- \ell \alpha} m_\alpha. 
$$
Thus, since $\alpha > 2pd/(d-2p)>dp/(d-p)$, the series $\sum_{\ell}  2^{\ell p} \mu (A_\ell)^{1-\frac{p}{d}}$ is convergent. We may then apply the dominated convergence theorem, we get from \eqref{eq:limsupAl} that 
$$
\limsup_n \E   \sum_{\ell \geq 0 }  \frac{L (  \cX \cap A_\ell  , \cY \cap A_\ell  ) }{n^{1-\frac{p}{d}}}  \le
 \beta_L \int f^{1-\frac{p}{d}}. 
$$
For the expectation of the second term on the right hand side of (\ref{eq:UP1}), we use Cauchy-Schwartz inequality, 
\begin{eqnarray*}
\lefteqn{\E  \left[T^p   \sum_{\ell \geq  0}     \1_{\cX  (A_\ell) +\cY( A_\ell)\neq 0}\big(1+|\cX  (A_\ell) -\cY( A_\ell)|\big) \right] }\\
  & \leq  &    \sum_{\ell \geq  0}    \sqrt {\E[T^{2p} ]} \sqrt{ \E \big( \1_{\cX  (A_\ell) +\cY( A_\ell)\neq 0}\big(1+|\cX  (A_\ell) -\cY( A_\ell)|\big)^2\big)  } \\
  &\le & \sqrt2  \sqrt {\E[T^{2p} ]}  \sum_{\ell \geq  0} \sqrt{ \mathbb P(\cX  (A_\ell) +\cY( A_\ell)\neq 0)
   + \E  \big[|\cX  (A_\ell) -\cY( A_\ell)|^2\big]  } \\
   &=& \sqrt2  \sqrt {\E[T^{2p} ]}  \sum_{\ell \geq  0} \sqrt{ 1-e^{-2n\mu(A_\ell)} + 2 n\mu(A_\ell)} \\
   &\le & 2 \sqrt {\E[T^{2p} ]}\, \sqrt{n}  \sum_{\ell \geq  0} \sqrt{\mu(A_\ell)},
\end{eqnarray*} 
where we have used $1-e^{-u}\le u$.
As above, Markov inequality leads to
$$
 \sum_{\ell \geq  0}  \sqrt{   \mu  (A_\ell)  }  \leq \sqrt{m_\alpha} \sum_{\ell \geq 0} 2^{-\ell \frac{\alpha}{2} } 
 <+\infty.
$$
Eventually we apply Lemma \ref{lem:ET} with $\gamma:=2p < 2pd/(d-2)<\alpha$ to upper bound $\E[T^{2p}]$. We get 
that  for some constant $c >0$ and all $n >0$, 
$$
n^{-1+\frac{p}{d} }  \E   \left[ T^p   \sum_{\ell \geq  0}     \1_{\cX  (A_\ell) +\cY( A_\ell)\neq 0}\big(1+|\cX  (A_\ell) -\cY( A_\ell)|\big) \right]  \leq   c n^{ - \frac 1 2 + \frac{p}{d} + \frac{p}{ \alpha}}.
$$
Since $\alpha > 2 dp / ( d -2p)$, the later and former terms tend to zero as $n$ tends to infinity.
The upper bound (\ref{eq:glupper}) is proved. \end{proof}

\section{Examples of bipartite functionals}
\label{sec:ex}

The minimal bipartite matching is an instance
of a bipartite Euclidean functional $M_1(\cX,\cY)$ over the multisets
$\mathcal X = \{X_1,\ldots,X_n\}$ and $\mathcal Y =
\{Y_1,\ldots,Y_n\}$. We may mention at least two other interesting
examples: the bipartite traveling salesperson problem over $\mathcal
X$ and $\mathcal Y$ is the shortest cycle on the multiset $\cX \cup
\cY$ such that the image of $\cX$ is $\cY$. Similarly, the bipartite
minimal spanning tree is the minimal edge-length spanning tree on $\cX
\cup \cY$ with no edge between two elements of $\cX$ or two elements
of $\cY$. 

\subsection{Minimal bipartite  matching}\label{sec:preparation}
  Fix $p>0$. Given two multi-subsets of $\R^d$
with the same cardinality, $\mathcal X=\{X_1,\ldots,X_n\}$ and
$\mathcal Y=\{Y_1,\ldots,Y_n\}$, the $p$-cost of the minimal bipartite
matching of $\mathcal X$ and $\mathcal Y$ is defined as
  $$ M_p(\mathcal X,\mathcal Y)=\min_{\sigma\in \mathcal S_n} \sum_{i=1}^{n} |X_i-Y_{\sigma(i)}|^p,$$
  where the minimum runs over all permutations of $\{1,\ldots,n\}$.
  It is useful to extend the definition to sets of different
  cardinalities, by matching as many points as possible: if $\mathcal
  X=\{X_1,\ldots,X_m\}$ and $\mathcal Y=\{Y_1,\ldots,Y_n\}$ and $m\le
  n$ then
   $$ M_p(\mathcal X,\mathcal Y)=\min_{\sigma} \sum_{i=1}^{m} |X_i-Y_{\sigma(i)}|^p,$$
   where the minimum runs over all injective maps from
   $\{1,\ldots,m\}$ to $\{1,\ldots,n\}$. When $n\le m$ the symmetric
   definition is chosen $ M_p(\mathcal X,\mathcal Y):=M_p(\mathcal Y,\mathcal X) $.

 The bipartite functional $M_p$ is obviously homogeneous of degree $p$, i.e. it satisfies $(\mathcal H_p)$.
The next lemma asserts that it is also verifies the  subadditivity
property $(\mathcal S_p)$. In the case $p=1$, this is the starting point of the paper \cite{BM}. 
\begin{lemma}\label{lem:partition}
For any $p >0$, the functional $M_p$ satisfies property \eqref{eq:Sp} with constant $C = 1/2$. More precisely, if $\mathcal X_1, \ldots,\mathcal X_k$ and $\mathcal Y_1,\ldots,\mathcal Y_k$ are multisets in a bounded subset $Q \subset \R^d$, then
  $$ M_p\Big(\bigcup_{i=1}^k\mathcal X_i , \bigcup_{i=1}^k\mathcal Y_i\Big) \le \sum_{i=1}^k  M_p(\mathcal X_i , \mathcal Y_i)
  + \frac{\mathrm{diam}(Q)^p}{2} \sum_{i=1}^k | \mathrm{card}(\mathcal
  X_i)- \mathrm{card}(\mathcal Y_i) |.$$
\end{lemma}
\begin{proof}
   It is enough to upper estimate of the cost of a particular matching
   of $\bigcup_{i=1}^k\mathcal X_i $ and $ \bigcup_{i=1}^k\mathcal
   Y_i$.  We build a matching of these multisets as follows. For each
   $i$ we choose the optimal matching of $\mathcal X_i$ and $\mathcal
   Y_i$. The overall cost is $ \sum_{i} M_p(\mathcal X_i , \mathcal
   Y_i)$, but we have left $ \sum_{i} | \mathrm{card}(\mathcal X_i)-
   \mathrm{card}(\mathcal Y_i)|$ points unmatched (the number of
   excess points).  Among these points, the less numerous species
   (there are two species: points from $\mathcal X_i$'s, and points
   from $\mathcal Y_i$'s) has cardinality at most $\frac{1}{2}
   \sum_{i} | \mathrm{card}(\mathcal X_i)- \mathrm{card}(\mathcal
   Y_i)|$. To complete the definition of the matching, we have to
   match all the points of this species in the minority.  We do this
   in an arbitrary manner and simply upper bound the distance between
   matched points by the diameter of $Q$.
\end{proof}

The regularity property is established next.
\begin{lemma}\label{lem:Rp}
For any $p >0$, the functional $M_p$ satisfies property \eqref{eq:Rp} with constant $C = 1$. 
\end{lemma}
\begin{proof}
Let $\cX,\cX_1,\cX_2,\cY,\cY_1,\cY_2$ be finite multisets contained in $Q = B(1/2)$. Denote by $x,x_1,x_2,y,y_1,y_2$ the cardinalities of the multisets and $a\wedge b$ for $\min(a,b)$. We start with an  optimal matching 
for $ M_p(\cX\cap \cX_2,\cY\cap \cY_2)$. It comprises $(x+x_2)\wedge (y+y_2)$ edges. We remove the ones
which have a vertex in $\cX_2$ or in $\cY_2$. There are at most $x_2+y_2$ of them, so we are left with
at least $\big((x+x_2)\wedge (y+y_2)-x_2-y_2\big)_+$ edges connecting points of $\cX$ to points of $\cY$.
We want to use this partial matching  in order to build a (suboptimal) matching of $\cX\cap \cX_1$ and
$\cY\cap \cY_1$. This requires to have globally $(x+x_1)\wedge (y+y_1)$ edges. Hence we need to add at most 
$$ (x+x_1)\wedge (y+y_1)-\big((x+x_2)\wedge (y+y_2)-x_2-y_2\big)_+$$
new edges. We do this in an arbitrary way, and simply upper bound their length by the diameter of $Q$.
To prove the claim it is therefore sufficient to prove the following inequalities for non-negative numbers:
\begin{equation}\label{eq:xy}
(x+x_1)\wedge (y+y_1)-\big((x+x_2)\wedge (y+y_2)-x_2-y_2\big)_+\le x_1+x_2+y_1+y_2.
\end{equation}
This is obviously equivalent to 
\begin{eqnarray*}
&&x+x_1 \le  x_1+x_2+y_1+y_2+\big((x+x_2)\wedge (y+y_2)-x_2-y_2\big)_+\\
& \mathrm{or} & y+y_1\le  x_1+x_2+y_1+y_2+ \big((x+x_2)\wedge (y+y_2)-x_2-y_2\big)_+.
\end{eqnarray*}
After simplification, and noting that $y_1\ge 0$ appears only on the right-hand side of the first inequation
(and the same for $x_1$ in the second one), it is enough to show that 
$$ x\wedge y \le  x_2+y_2+ \big((x+x_2)\wedge (y+y_2)-x_2-y_2\big)_+.$$
This is obvious, as by definition of the positive part, $ x\wedge y \le  x_2+y_2+ \big((x\wedge y)-x_2-y_2\big)_+.$
\end{proof}

 \subsection{Bipartite traveling salesperson tour}

 Fix $p>0$. Given two multi-subsets of $\R^d$ with the same cardinality, $\mathcal X=\{X_1,\ldots,X_n\}$ and
$\mathcal Y=\{Y_1,\ldots,Y_n\}$, the $p$-cost of the minimal bipartite
traveling salesperson tour of $(\mathcal X,\mathcal Y)$ is defined as
  $$T_p(\mathcal X,\mathcal Y)=\min_{(\sigma,\sigma') \in S_n ^2  } \sum_{i=1}^{n} |X_{\sigma(i)} -Y_{\sigma'(i)}|^p +  \sum_{i=1}^{n-1} |Y_{\sigma'(i)} -X_{\sigma(i+1)}|^p + |Y_{\sigma'(n)} -X_{\sigma(1)}|^p ,$$
  where the minimum runs over all pairs of permutations of $\{1,\ldots,n\}$.
  We extend the definition to sets of different cardinalities, by completing the longest possible bipartite tour : if $\mathcal
  X=\{X_1,\ldots,X_m\}$ and $\mathcal Y=\{Y_1,\ldots,Y_n\}$ and $m\le
  n$ then
   $$ T_p(\mathcal X,\mathcal Y)=\min_{(\sigma,\sigma')}  \sum_{i=1}^{m} |X_{\sigma(i)} -Y_{\sigma'(i)}|^p +  \sum_{i=1}^{m-1} |Y_{\sigma'(i)} -X_{\sigma(i+1)}|^p + |Y_{\sigma'(m)} -X_{\sigma(1)}|^p $$
   where the minimum runs over all pairs $(\sigma, \sigma')$, with $\sigma \in S_m$ and $\sigma'$ is an injective maps from
   $\{1,\ldots,m\}$ to $\{1,\ldots,n\}$. When $n\le m$ the symmetric
   definition is chosen $ T_p(\mathcal X,\mathcal Y):=T_p(\mathcal Y,\mathcal X) $. This traveling salesperson functional is an instance of a larger class of functionals that we now describe.

 \subsection{Euclidean combinatorial optimization over bipartite graphs}
 
 \label{subsec:combopt}
 
For integers $m,n$, we define $[n] = \{1, \cdots n\}$ and $[n]_m = \{m +1 , \cdots , m + n\}$.  Let $\mathcal B_n$ be the set of bipartite graphs with common vertex set $([n], [n]_n)$ : if $G \in \mathcal B_n$, the edge set of $G$ is contained is the set of pairs $\{i , n+j\}$, with $i, j \in [n]$. 

We should introduce some graph definitions. If $G_1 \in \mathcal B_n$ and $G_2 \in \mathcal B_m$ we define $G_1 + G_2$ as the graph in $\mathcal B_{n+m}$ obtained by the following rule : if $\{i , n + j\}$ is an edge of $G_1$ then $\{i , n +m+ j\}$ is an edge of $G_1 + G_2$, and if $\{i , m + j\}$ is an edge of $G_2$ then $\{n+i , 2n +m+ j\}$ is an edge of $G_1 + G_2$.  Finally, if $G \in \mathcal B_{n+m}$, the restriction $G'$ of $G$ to $\mathcal B_n$ is the element of $\mathcal B_n$ defined by the following construction rule:
 if $\{i , n + m + j\}$ is an edge of $G$ and $(i,j) \in [n]^2$ then add $\{i, n+ j \}$ as an edge of $G'$.

We consider a collection of subsets $\mathcal G_n \subset \mathcal B_n$ with the following properties, there exist constants $\kappa_0, \kappa \geq 1$ such that for all integers $n,m$, 

\begin{itemize}

\item[(A1)] {\em (not empty)} If $n \geq \kappa_0$, $\mathcal G_n$ is not empty. 

\item[(A2)]
{\em  (isomorphism)} If $G \in \mathcal G_n$ and $G'  \in \mathcal B_n$ is isomorphic to $G$ then $G' \in \mathcal G_n$.

\item[(A3)] {\em (bounded degree)} If $G \in \mathcal G_n$, the degree of any vertex is at most $\kappa$.

\item[(A4)]
{\em (merging)} If $G \in \mathcal G_n$ and $G' \in \mathcal G_m$, there exists $G'' \in \mathcal G_{n+m}$ such that $G + G'$ and $G''$ have all but at most $\kappa$ edges in common. For $1 \leq m <  \kappa_0$, it also holds if $G'$ is the empty graph of $\mathcal B_m$.

\item[(A5)]
{\em (restriction)}  Let $G \in \mathcal G_n$ and $\kappa_0+1\leq n$ and $G'$ be  the restriction of $G$ to $\mathcal B_{n-1} $. Then  there exists $G'' \in \mathcal G_{n-1}$  such that  $G'$  and $G''$ have all but at most $\kappa$ edges in common.
\end{itemize}

 If $|\mathcal X |= |\mathcal Y| = n$, we define 
 $$
 L ( \mathcal X , \mathcal Y ) = \min_{G \in \mathcal G_n} \sum_{(i,j)\in [n]^2: \{i , n+ j \} \in  G } | X_i - Y_{j} |^p . 
 $$
With the convention that the minimum over an empty set is $0$. Note that the isomorphism property implies that $ L ( \mathcal X , \mathcal Y )  =  L (  \mathcal Y , \mathcal X )$.   If $m = |\mathcal X |  \leq |\mathcal Y| = n$, we define 
 \begin{equation}
 \label{eq:combopt}
 L ( \mathcal X , \mathcal Y ) = \min_{(G,\sigma)} \sum_{(i,j)\in [m]^2: \{i , m+j \} \in  G } | X_i - Y_{\sigma(j)} |^p,
\end{equation}
where the miminum runs over all pairs $(G, \sigma)$, $G  \in \mathcal G_m$ and $ \sigma$ is an injective maps from
   $\{1,\ldots,m\}$ to $\{1,\ldots,n\}$. When $n\le m$ the symmetric
   definition is chosen $ L(\mathcal X,\mathcal Y):=L(\mathcal Y,\mathcal X) $.

 \medskip
 
The case of bipartite matchings is recovered by choosing $\mathcal G_n$ as the set of graphs in $\mathcal B_n$ where all vertices have degree $1$. We then have $\kappa_0 =1$ and $\mathcal G_n$ satisfies the merging  property with $\kappa = 0$. It also satisfies the restriction property with $\kappa = 1$.  The case of the traveling salesperson tour is obtained by choosing $\mathcal G_n$ as the set of connected graphs in $\mathcal B_n$ where all vertices have degree $2$, this set is non-empty for $n \geq \kappa_0 = 2$. Also this set $\mathcal G_n$ satisfies the merging property with $\kappa = 4$ (as can be checked by edge switching). The restriction property follows by merging strings into a cycle. 

For the minimal bipartite spanning tree, we choose $\mathcal G_n$ as the set of connected trees of $[2n]$ in $\mathcal B_n$. It satisfies  the restriction property and the merging property with $\kappa = 1$. For this choice, however, the maximal degree is not bounded uniformly in $n$.  We could impose artificially  this condition by defining $\mathcal G_n$ as the set of connected graphs in $\mathcal B_n$ with maximal degree bounded by $\kappa \geq 2$. We would then get the minimal bipartite spanning tree with maximal degree bounded by $\kappa $. It is not hard to verify that the corresponding functional satisfies all the above properties.

Another interesting example is the following. Fix an integer $r \geq 2$. Recall that a graph is $r$-regular if the degree of all its vertices is equal to $r$. We may define $\mathcal G_n$ as the set of $r$-regular connected graphs in $\mathcal B_n$. This set is not empty for $n \geq \kappa_ 0 = r$. It satisfies  the first part of the merging property (A4) with $\kappa = 4$. Indeed, consider two $r$-regular graphs $G$, $G'$, and take any edge $e = \{x,y\} \in G$ and $e' = \{x',y'\} \in G'$. The merging property holds with $G''$, the graph obtained from $G+G'$ by switching $(e,e')$ in $(\{x,y'\},\{x',y\})$. Up to increasing the value of $\kappa$, the second part of the merging property is also satisfied. Indeed, if $n$ is large enough, it is possible to find $r m < r \kappa_0 = r^2$ edges $e_{1}, \cdots, e_{rm}$ in $G$ with no-adjacent vertices. Now, in $G''$, we add $m$ points from each species, and replace the edge $e_{r i + q} = \{x , n + y\}$, $1 \leq i \leq m$, $0 \leq q < r$, by two edges : one between $x$ and the $i$-th point of the second species, and one between $y$ and the $i$-th point of the first species. $G''$ is then a connected $r$-regular graph in $\mathcal B_{n+m}$ with all but at most $2r^2$ edges in common with $G$. Hence, by taking $\kappa$ large enough, the second part of the merging property holds.

Checking the restriction property (A5) for $r$-regular graphs requires a little more care. Let $r = \kappa_0 +1\leq n$ and consider the restriction $G_1$ of $G \in \mathcal B_n$ to $\mathcal B_{n-1}$. Our goal is to show that by modifying a small number of edges of $G_1$, one can obtained a connected $r$-regular bipartite
graph on $\mathcal B_{n-1}$. We first explain how to turn $G_1$ into a possibly non-connected $r$-regular graph. Let us observe that $G_1$ was obtained from
$G$ by deleting one vertex of each spieces and the edges to which these points belong. Hence $G_1$ has vertices of degree $r$, and vertices of degree $r-1$
($r$ blue and $r$ red vertices if the removed points did not share an edge, only $r-1$ points of each spieces if the removed points shared an edge). 
In any case $G_1$ has at most $2r$ connected components and $r$ vertives of each color with one edge missing.
The simplest way to turn $G_1$ into a $r$ regular graph is to connect each blue vertex missing an edge with a red vertex missing an edge. However this is 
not always possible as these vertices may already be neighbours in $G_1$ and we do not allow multiple edges. However given a red vertex $v_R$ and a blue vertex
$v_B$ of degree
$r-1$ and provided $n-1 > 2r^2$ there exists a vertex $v$ in $G_1$ which is at graph distance at least 3 from $v_B$ and $v_R$. Then open up an edge to which
$v$ belongs and connect its end-points to $v_R$ and $v_B$ while respecting the bipartite structure. In the new graph $v_B$ and $v_R$ have degree $r$.
Repeating this operation no more than $r$ times turns $G_1$ into a $r$ regular graphs with at most as many connected components (and the initial and the final 
graph differ by at most $3r$ edges). Next we apply the merge operation at most $2r-1$ times in order to glue together the connected componented (this leads
to modifying at most $4(2r-1)$ edges. As a conclusion, provided we choose $\kappa_0>2r^ 2$, the restriction property holds for $\kappa=11r$.

\medskip

We now come back to the general case. From the definition, it is clear that $L$ satisfies the property $\eqref{eq:Hp}$. We are going to check that it also satisfies properties \eqref{eq:Sp} and  \eqref{eq:Rp}. 
\begin{lemma}\label{le:Sp}
Assume (A1-A4). For any $p >0$, the functional $L$ satisfies property \eqref{eq:Sp} with constant $C  = ( 3 + \kappa_0 ) \kappa /2 $. 
\end{lemma}

\begin{proof} The proof of is an extension of the proof of Lemma \ref{lem:partition}.  We can assume without loss of generality $k \geq 2$. Let $\mathcal X_1, \ldots,\mathcal X_k$ and $\mathcal Y_1,\ldots,\mathcal Y_k$ be multisets in $Q  = B(1/2)$.  For ease of notation, let $x_i = |\mathcal X_i|$,  $y_i = |\mathcal Y_i|$ and   $n =  \sum_{i=1}^k x_i  \wedge \sum_{i=1}^k y_i $. If $n < \kappa_0$, then from the bounded degree property (A3),  
$$L \Big(\bigcup_{i=1}^k\mathcal X_i , \bigcup_{i=1}^k\mathcal Y_i\Big) \leq  n \kappa   \leq \kappa \kappa_0. $$
If $n \geq \kappa_0$,  it is enough to upper bound the cost for an element $G$ in $\mathcal G_n$. For each $1 \leq i \leq k$, if $n_i = x_i \wedge y_i \geq \kappa_0$, we consider the element $G_i$ in $\mathcal G_{n_i}$ which reaches the minimum cost of $L(\mathcal X_i, \mathcal Y_i)$.  From the merging property (A4), there exists  $G'$ in $\mathcal G_{\sum_{i} \ind_{n_i \geq \kappa_0} n_i}$ whose total cost is at most 
$$
L' :=  \sum_{i} L(\mathcal X_i , \mathcal Y_i) + \kappa k . 
$$
It remains at most $ \sum_{i} \kappa_0 +  |x_i - y_i|$ vertices that have been left aside. The less numerous species has cardinal  $ m_0 \leq m  = ( \sum_{i} \kappa_0 +  |x_i - y_i| ) / 2 $. If $m_0 \geq \kappa_0$, from the non-empty property (A1), there exists a graph $G'' \in \mathcal G_{m_0}$ that minimizes the cost of the vertices that have been left aside. From the merging and bounded degree properties, we get 
$$
L \Big(\bigcup_{i=1}^k\mathcal X_i , \bigcup_{i=1}^k\mathcal Y_i\Big) \leq L'  +      \kappa   +   \kappa  m   \leq \sum_{i} L(\mathcal X_i , \mathcal Y_i)  + \frac{\kappa}{2} \sum_i \left( 3    +  \kappa_0 +  |x_i - y_i | \right). 
$$
If $ m_0 < \kappa_0$, we apply to $G'$ the merging property with the empty graph :  there exists an element $G$ in $\mathcal G_n$ whose total cost is at most  
$$
L \Big(\bigcup_{i=1}^k\mathcal X_i , \bigcup_{i=1}^k\mathcal Y_i\Big) \leq  L'  +    \kappa  \leq   \sum_{i} L(\mathcal X_i , \mathcal Y_i) +   (k +1 ) \kappa .
$$
We have proved that property \eqref{eq:Sp} is satisfied for $ C =( 3 + \kappa_0 )  \kappa  /2$. 
\end{proof}

\begin{lemma}
\label{le:Rp}
Assume (A1-A5). For any $p >0$, the functional $L$ satisfies property \eqref{eq:Rp} with constant $C =C(\kappa, \kappa_0)  $. 
\end{lemma}

\begin{proof}
  Let $\cX,\cX_1,\cX_2,\cY,\cY_1,\cY_2$ be finite multisets contained in $B(1/2) = Q$. Denote by $x,x_1,x_2,y,y_1,y_2$ the cardinalities of the multisets.
  As a first step, let us prove that   
  \begin{equation}\label{eq:RpsansXY2}
L(\cX\cup\cX_1 ,\cY\cup\cY_1) \leq L(\cX,\cY) + C ( x_1 + y_1).
\end{equation}
By induction, it is enough to deal with the cases $(x_1,y_1)=(1,0)$ and $(x_1,y_1)=(0,1)$.
Because of our symmetry assumption, our task is to prove that 
 \begin{equation}\label{eq:RpsansY1XY2}
  L(\cX\cup\{a\} ,\cY) \leq L(\cX,\cY) + C .
  \end{equation}
If   $ \card(\cY)\le \card(\cX)$, then the latter is obvious: choose an optimal graph for $L(\cX,\cY)$ and use it to upper estimate 
$ L(\cX\cup\{a\} ,\cY)$.
Assume on the contrary that $\card(\cY)\ge \card(\cX)+1$. Then there exists $\cY'\subset \cY$ with  $ \card(\cY')= \card(\cX)$ and
$L(\cX,\cY')$. Let $b\in \cY\setminus \cY'$. In order to establish \eqref{eq:RpsansY1XY2}, it is enough to show that
$$  L(\cX\cup\{a\} ,\cY'\cup\{b\}) \leq L(\cX,\cY') + C ,$$
but this is just an instance of the subadditivity property. Hence \eqref{eq:RpsansXY2} is established.

\medskip
   In order to prove the regularity property, it remains to show  that 
\begin{equation}\label{eq:RpsansXY1}
L(\cX ,\cY) \leq L(\cX\cup \cX_2,\cY\cup \cY_2) + C ( x_2 + y_2).
\end{equation}
Again, using induction and symmetry, it is sufficient to establish
\begin{equation}\label{eq:RpsansY2XY1}
  L(\cX,\cY) \leq L(\cX\cup\{a\},\cY) + C .
  \end{equation}
If $\card(\cX)\wedge\card(cY)<\kappa_0$, then by the bounded degree property $ L(\cX,\cY) \leq \kappa\kappa_0 \diam(Q)^p$  and we are done.
Assume next that  $\card(\cX),\card(\cY)\ge \kappa_0$.
Let us consider an optimal graph for $  L(\cX\cup\{a\},\cY) $. If $a$ is not a vertex of this graph (which forces
$\card(\cX)\ge \card(\cY)$) then one can use the same graph to upper estimate $ L(\cX,\cY)$ and obtain \eqref{eq:RpsansY2XY1}.
Assume on the contrary that $a$ is a vertex of this optimal graph.  Let us distinguish two cases:
if $\card(\cX)\ge \card(\cY)$, then in the optimal graph for $  L(\cX\cup\{a\},\cY) $, at least a point $b\in \cX$ is not used.
Consider the isomorphic graph obtained by replacing $a$ by $b$ while the other points remain fixed 
(this leads to the deformation of the edges out of $a$. There are at most $\kappa$ of them by the bounded degree assumption).
This graph can be used to upper estimate $ L(\cX,\cY)$, and gives
$$ L(\cX,\cY)\le  L(\cX\cup\{a\},\cY) + \kappa \,\mathrm{diam}(Q)^p .$$
The second case is when $a$ is used but $\card(\cX)+1\le \card(\cY)$. Actually, the optimal graph for $ L(\cX\cup\{a\},\cY)$
uses all the points of $\cX\cup\{a\}$ and of a subset of same cardinality $\cY'\subset \cY$.
Choose an element $b$ in $\cY'$. Then $\cY''=\cY'\setminus\{b\}$ has the same cardinality as $\cX$.
Obviously $L(\cX\cup\{a\},\cY)=L(\cX\cup\{a\},\cY''\cup\{b\})$. Consider the corresponding optimal bipartite graph.
By the restriction property, if we erase $a$ and $b$ and their edges, we obtain a bipartite graph on $(\cX,\cY'')$ 
which differs from an admissible graph of our optimization problem by at most $\kappa$ edges. Using this new graphs
yields
$$ L(\cX,\cY)\le \kappa\, \mathrm{diam}(Q)^p + L(\cX\cup\{a\},\cY''\cup\{b\})= \kappa\, \mathrm{diam}(Q)^p + L(\cX\cup\{a\},\cY).$$
This concludes the proof.\end{proof}

\section{Lower bounds, lower limits}\label{sec:lower}

\subsection{Uniform distribution on a set}

In order to motivate the sequel, we start with the simple case where
$f$ is an indicator function. The lower bound is then a direct
consequence of Theorem~\ref{th:cube-poisson} and
Theorem~\ref{th:upper}.

\begin{theorem} \label{th:lowerset} 
 Let $d>2p>0$. Let $L$ be a bipartite functional on $\R^d$ satisfying the properties $(\mathcal H_p)$,
   $(\mathcal R_p)$, $(\mathcal S_p)$.
 Let $\Omega\subset \R^d$ be a bounded set with positive Lebesgue measure.  Then
$$ \lim_{n\to \infty} \frac{\E L(n\1_\Omega)}{n^{1-\frac{p}{d}}}= \beta_L \mathrm{Vol}(\Omega).$$
\end{theorem}
\begin{proof}
   Theorem~\ref{th:upper} gives directly $\limsup \E
   L(n\1_\Omega)/n^{1-\frac{p}{d}}\le \beta_L \mathrm{Vol}(\Omega)$.
   By translation and dilation invariance, we may assume without loss
   of generality that $\Omega\subset [0,1]^d$.  Let $\Omega_c:=
   [0,1]^d \setminus \Omega$. Applying Proposition~\ref{prop:partition} for
   the partition $[0,1]^d=\Omega \cup \Omega_c$, gives after division
   by $n^{1-p/d}$
$$ \frac{\E L\big(n\1_{[0,1]^d}\big)}{n^{1-\frac{p}{d}}}- \frac{\E L\big(n\1_{\Omega_c}\big)}{n^{1-\frac{p}{d}}}
\le \frac{\E L\big(n\1_{\Omega}\big)}{n^{1-\frac{p}{d}}} + 3C \mathrm{diam}([0,1]^d)
n^{\frac{p}{d}-\frac12}
\Big(\mathrm{Vol}(\Omega)^{\frac12}+\mathrm{Vol}(\Omega_c)^{\frac12}\Big).
$$
Since $d>2p$, letting $n$ go to infinity gives
\begin{eqnarray*}
   \liminf_n \frac{\E L\big(n\1_{\Omega}\big)}{n^{1-\frac{p}{d}}} &\ge& \lim_n  \frac{\E L\big(n\1_{[0,1]^d}\big)}{n^{1-\frac{p}{d}}}- \limsup_n \frac{\E L\big(n\1_{\Omega_c}\big)}{n^{1-\frac{p}{d}}}\\
   &\ge&\beta_L-\beta_L \mathrm{Vol}(\Omega_c)=\beta_L \mathrm{Vol}(\Omega),
\end{eqnarray*}
where we have used Theorem~\ref{th:cube-poisson} for the limit and
Theorem~\ref{th:upper} for the upper limit.
\end{proof}
The argument of the previous proof relies on the fact that the
quantity $\lim n^{1-p/d}\E L(n\1_\Omega)
=\beta_L\mathrm{Vol}(\Omega)$ is in a sense additive in $\Omega$.
This line of reasoning does not pass to functions since $f\mapsto
\int f^{1-p/d}$ is additive only for functions with disjoint
supports. The lower limit result requires more work for general
densities.

\subsection{Lower limits for matchings}

In order to establish a tight estimate on the lower limit, it is
natural to try and reverse the partition inequality given in
Proposition~\ref{prop:partition}. This is usually more difficult and there does not exist a general method to perform this lower bound. We shall first restrict our attention to the case of the matching functional $M_p$ with $p >0$, we define in this subsection
$$
L = M_p.
$$

\subsubsection{Boundary functional}

Given a matching on the unit
cube, one needs to infer from it matchings on the subcubes of a dyadic
partition and to control the corresponding costs. The main difficulty
comes from the points of a subcube that are matched to points of
another subcube. In other words some links of the optimal matching
cross the boundaries of the cells.  As in the book by Yukich \cite{Y},
a modified notion of the cost of a matching is used in order to
control the effects of the boundary of the cells of a partition. Our
argument is however more involved, since the good bound
\eqref{eq:good-bound} used by Yukich is not available for the
bipartite matching.  \medskip

    We define 
    $$
    q =  2^{p-1} \wedge 1.
    $$
    
Let $S\subset \mathbb R^d$ and $\varepsilon \ge 0$. Given multisets
$\cX=\{X_1,\ldots,m\}$ and $\cY=\{Y_1,\ldots,Y_n\}$ included in $S$ we
define the penalized boundary-matching cost as follows
\begin{eqnarray}\label{eq:Lboundary}
   \lefteqn{L_{\partial S,\varepsilon}(X_1,\ldots,X_m;Y_1,\ldots,Y_n)} \\
   &=& \min_{A,B,\sigma}\left\{ 
      \sum_{i\in A} |X_i-Y_{\sigma(i)}|^p + \sum_{i\in A^c}q  \Big( d(X_i,\partial S)^p +\varepsilon^p \Big) 
      +\sum_{j\in B^c} q \Big( d(Y_j,\partial S)^p +\varepsilon^p \Big) \right\}, \nonumber
\end{eqnarray}
where the minimum runs over all choices of subsets
$A\subset\{1,\ldots, m\}$, $B\subset\{1,\ldots, n\}$  with the same
cardinality and all bijective maps $\sigma:A\to B$.  When $\varepsilon=0$
we simply write $L_{\partial S}$. Notice that in our
definition, and contrary to the definition of optimal matching, all
points are matched even if $m\neq n$.  If $\cX$ and $\cY$ are
independent Poisson point processes with intensity $\nu$ supported in
$S$ and with finite total mass, we write $L_{\partial S,
  \varepsilon}(\nu)$ for the random variable $L_{\partial S,
  \varepsilon}(\cX,\cY)$.

The main interest of the notion of boundary matching is that it allows
to bound from below the matching cost on a large set in terms of
contributions on cells of a partition. The following Lemma establishes a superadditive property of $L_{\partial S}$ and it can be
viewed as a counterpart to the upper bound provided by
Proposition~\ref{prop:partition}.
     
     \begin{lemma}\label{prop:part2}
   Assume $L = M_p$.     Let $\nu$ be a finite measure on $\mathbb R^d$ and consider a
        partition $Q=\cup_{P\in\mathcal P} P$ of a subset of $\mathbb
        R^d$. Then
       $$ \mathrm{diam}(Q)^p \sqrt{2\nu(\mathbb R^d)} + \E L(\nu)\ge  \E L_{\partial Q} (\ind_{Q} \cdot \nu ) \geq \sum_{P\in\mathcal P} \E L_{\partial P} (\1_P\cdot
       \nu).$$
    \end{lemma}
    \begin{proof}
       Let $\mathcal X=\{X_1,\ldots,X_m\},\mathcal
       Y=\{Y_1,\ldots,Y_n\}$ be multisets included in $Q$ and $\cX'=\{X_{m+1},\ldots,X_{m+m'}\}$, $\cY'=\{Y_{n+1},\ldots,Y_{n+n'}\}$ be multisets included in $Q^c$. By considering 
an optimal matching of $\cX\cup \cX'$ and $\cY \cup \cY'$, we have the lower bound 
      $$     \mathrm{diam}(Q)^p   |m+m'-n - n'|+ L(\cX\cup \cX',\cY\cup \cY') \geq L_{\partial Q}  (\cX,\cY). $$  
Indeed, if $1 \leq i \leq m$ and a pair $(X_i, Y_{n+j})$, is matched then $|X_i -  Y_{n+j}| \geq d ( X_i , \partial Q)$ and similarly for a pair $(X_{m+i}, Y_{j})$, with $1 \leq j \leq n$, $|X_{m+i} -  Y_{j}| \geq d ( Y_j , \partial Q)$. The term $\mathrm{diam}(Q)^p   |m+m'-n - n'| $ takes care of the points of $\cX \cup \cY$ that are not matched in the optimal matching of   $\cX\cup \cX'$ and $\cY \cup \cY'$.  We apply the above inequality to $\cX$, $\cY$ independent Poisson processes of intensity $\ind_Q \cdot \nu$, and $\cX'$, $\cY'$, two independent Poisson processes of intensity $\ind_{Q^c} \cdot \nu$, independent of $(\cX,\cY)$. Then $\cX \cup \cX'$, $\cY \cup \cY'$ are independent Poisson processes of intensity $\nu$. Taking expectation and
         bounding the average of the difference of cardinalities in
         the usual way, we obtain the first inequality.

Now, the second inequality will follow from the superadditive property of the boundary functional: 
     \begin{equation} \label{eq:SLb}    L_{\partial Q} (\cX,\cY) \geq  \sum_{P \in\mathcal P}  L_{\partial P} ( \cX \cap P,\cY \cap P).
     \end{equation} 
     This is proved as follows. Let $(A,B,\sigma)$ be an optimal triplet for $L_{\partial Q} (\cX,\cY)$:
     $$
     L_{\partial Q} (\cX,\cY)  = \sum_{i\in A} |X_i-Y_{\sigma(i)}|^p + \sum_{i\in A^c}q  d(X_i,\partial Q)^p  
      +\sum_{j\in B^c} q  d(Y_j,\partial Q)^p.
     $$ 
 If $x \in Q$, we denote by $P(x)$ the unique $P \in \cP$ that contains $x$. If $P(X_i)=P(Y_{\sigma(i)})$ we leave the term
      $|X_i-Y_{\sigma(i)}|$ unchanged. On the other hand if $P(X_i)\neq P(Y_{\sigma(i)})$, from H\"older's inequality, 
       $$ |X_i-Y_{\sigma(i)}|^p \ge q \, d(X_i,\partial P(X_i))^p  + q \, d(Y_{\sigma(i)}, \partial P(Y_{\sigma(i)}))^p .$$
       Eventually, we apply the inequality
         $$ d(x,\partial Q) \geq  d(x,\partial P(x))$$
         in order to take care of the points in $A^c \cup B^c$. Combining
         these inequalities and grouping the terms according to the
         cell $P\in\mathcal P$ containing the points, we obtain that
         \begin{eqnarray*}
   L_{\partial Q} (\cX,\cY) &\ge &  \sum_{P\in\mathcal P}\left(\sum_{i\in A;\; X_i\in P, Y_{\sigma(i)}\in P} |X_i-Y_{\sigma(i)}| ^p  + \sum_{i\in A;\; X_i\in P, Y_{\sigma(i)}\notin P}q \, d(X_i,\partial P)^p   \right.\\
  && \left.            +   \sum_{i \in A^c;\; X_i\in P} q \, d(X_i,\partial P)^p +
               \sum_{j\in B ;\; Y_j \in P,\; j\not\in \sigma(\{i; \;X_i\in P\})} q \, d(Y_j,\partial P) ^p  +  \sum_{j \in B^c;\; Y_j\in P} q \, d(Y_j,\partial P)^p\right)\\
            &\ge & \sum_{P\in\mathcal P} L_{\partial P}(\cX\cap P,\cY\cap P),
         \end{eqnarray*}
      and  we have obtained the inequality \eqref{eq:SLb}.
      \end{proof}

      The next lemma will be used to reduce to uniform distributions
      on squares.

\begin{lemma}\label{lem:coupling2}
   Assume $L = M_p$. Let $\mu,\mu'$ be two probability measures on $\R^d$ with supports in $Q$ and $ n >0$. Then 
$$   \E L_{\partial Q}(n\mu)\le  \E L_{\partial Q}(n\mu') + 4  n \,\mathrm{diam}(Q)^p\, d_\mathrm{TV}(\mu,\mu').$$
Consequently, if $f$ is a nonnegative locally integrable function on
   $\R^d$, setting $\alpha=\int_Q f/\mathrm{vol}(Q)$, it holds
  $$ \E L_{\partial Q}(nf \1_Q) \le \E L_{\partial Q}(n\alpha \1_Q) + 2  n\, \mathrm{diam}(Q)^p \, \int_Q|f(x)-\alpha|\, dx.$$
  \end{lemma}

\begin{proof}
The functional $ L_{\partial Q}$ satisfies a slight modification of property $(\mathcal R_p)$ :  for all multisets $\mathcal X,\mathcal Y,
\mathcal X_1,\mathcal Y_1,\mathcal X_2,\mathcal Y_2$ in $Q$, it holds 
\begin{equation*}
L_{\partial Q}(\mathcal X\cup \cX_1,\cY \cup\cY_1)\le L_{\partial Q}(\mathcal X\cup \cX_2,\cY \cup\cY_2) +  \mathrm{diam}(Q) ^p \big( \mathrm{card}(\cX_1)+
\mathrm{card}(\cX_2)+\mathrm{card}(\cY_1)+\mathrm{card}(\cY_2)\big).
\end{equation*}  
Indeed, we start from an optimal boundary matching of $L_{\partial Q}(\mathcal X\cup \cX_2,\cY \cup\cY_2)$, we match to the boundary the points of $(\cX, \cY)$ that are matched to a point in $(\cX_2, \cY_2)$. There are at most $\mathrm{card}(\cX_2) + \mathrm{card}(\cY_2)$ such points.  Finally we match all points of $(\cX_1,\cY_1)$ to the boundary and we obtain a suboptimal boundary matching of $L_{\partial Q}(\mathcal X\cup \cX_1,\cY \cup\cY_1)$. This establishes the above inequality. The statements follow then from the proofs of Proposition \ref{prop:approx} and Corollary \ref{cor:couplingaverage}. 
   \end{proof}

We will need an asymptotic for the boundary matching for the uniform
distribution on the unit cube. Let $Q=[0,1]^d$ and denote $$\bar L_{\partial Q}(n)=\E L_{\partial Q}(n\1_Q).$$
\begin{lemma}
   \label{lem:liminfBM}
   Assume $L = M_p$ and $0 < p<d/2$, then
$$ \lim_{n\to \infty} \frac{\bar L_{\partial Q}(n)}{n^{1-\frac{p}{d}}} = \beta'_L,$$
where $\beta'_L >0$ is a constant depending on $p$ and $d$.
\end{lemma}

\begin{proof}
Let $m\ge 1$ be an integer. We consider a dyadic partition $\mathcal P$ of $Q$ into $m^d$ cubes of size $1/m$. Then, Lemma \ref{prop:part2}  applied  to the measure $n\1_{[0,1]^d}(x) \, dx$ gives 
      $$\bar L_{\partial Q} (n) \geq \sum_{q\in \mathcal P} \E L_{\partial q} (n \1_{q\cap[0,1]^d}).$$
   However by scale and translation invariance, for any $q \in \mathcal P$ we have $\E L_{\partial q} (n \1_{q\cap[0,1]^d}) = m^{-p} \E L_{\partial Q}  (n m^{-d} \1_Q)$. It follows that 
   $$
   \bar L_{\partial Q} (n) \geq m^{d-p}  \bar L_{\partial Q}  (n m^{-d}). 
   $$
 The proof is then done as in Theorem \ref{th:cube-poisson} where superadditivity here replaces subadditivity there. 
\end{proof}

\subsubsection{General absolutely continuous measures}

We are ready to state and prove
\begin{theorem} \label{th:lower}    Assume $L = M_p$ and $0 < p<d/2$. Let $f: \R^d \to \R^+$
   be an integrable function.  Then $$ \liminf_{n}
   \frac{\E L(n f )}{n^{1 - \frac{p}{d}}}\geq \beta'_L \int_{\R^d}
   f^{1 - \frac{p}{d}}.$$
   \end{theorem}

\begin{proof}
Assume first that the support of $f$ is bounded.
By a scaling argument, we may assume that the support of $f$ is
included in $Q=[0,1]^d$.  The proof is now similar to the one of
Theorem~\ref{th:upper}.  For $\ell \in\mathbb N$, we consider the
partition $\mathcal P_\ell$ of $[0,1]^d$ into $2^{\ell d}$ cubes of
side-length $2^{-\ell}$. Let $k\in \mathbb N^*$ to be chosen later.
For $P\in \mathcal P_k$, $\alpha_P$ denotes the average of $f$ over
$P$.  Applying Lemma~\ref{prop:part2}, Lemma~\ref{lem:coupling2} and
homogeneity, we obtain
\begin{eqnarray*}
2 d^{\frac p 2}   \sqrt{n\int f}+\E L(nf)\;  \geq \; \E L_{\partial Q} (nf) &\ge& \sum_{P\in \mathcal P_k} \E L_{\partial P} (nf \1_P) \\
   &\ge& \sum_{P\in \mathcal P_k}\left(  \E L_{\partial P} (n \alpha_P \1_P)- 2 n d^{\frac p 2}   \, 2^{-kp} \int_{P} |f-\alpha_P| \right)\\
   &=& \sum_{P\in \mathcal P_k}\left(2^{-kp}  \E L_{\partial Q} (n \alpha_P  2^{-kd}\1_Q)- 2 nd^{\frac p 2} \, 2^{-kp} \int_{P}|f-\alpha_P| \right).
\end{eqnarray*}
Setting as before $f_k= \sum_{P\in \mathcal P_k} \alpha_P\1_P$ and
$h(t)=\bar L_{\partial Q} (t)/t^{\frac{d-1}{d}}$ where $\bar L_{\partial Q}(t)=\E L_{\partial Q}
(t\1_Q)$, the previous inequality reads as
$$  2 n^{\frac{p}{d}-\frac12} d^{\frac p 2}  \sqrt{ \int f} + \frac{\E L(nf)}{n^{1 - \frac{p}{d}}} \geq \E L_{\partial Q} (nf) \ge \int h(n2^{-kd} f_k)f_k^{1 - \frac{p}{d}}- 2 d ^{\frac p 2 }\, n^{\frac{p}{d}}2^{-k p }
\int|f-f_k|.$$ As in the proof of Theorem~\ref{th:upper} we may choose
$k=k(n)$ depending on $n$ in such a way that $\lim_n k(n)=+\infty$,
$\lim_n n^{1/d}2^{-k(n)}=+\infty$ and $\lim_n n^{\frac{1}{d}}2^{-k(n)}
\big( \int|f-f_{k(n)}| \big)^{\frac 1 p}=0$. For such a choice, since $\liminf_{t\to +\infty }
h(t)\ge\beta'_L$ by Lemma~\ref{lem:liminfBM} and a.e. $\lim_k f_k=f$,
Fatou's lemma ensures that
$$\liminf_n  \int h(n2^{-k(n)d} f_{k(n)})f_{k(n)}^{1 - \frac{p}{d}} \ge \liminf_n  
\int_{\{f>0\}} h(n2^{-k(n)d} f_{k(n)})f_{k(n)}^{1 - \frac{p}{d}} \ge
\beta'_L \int f^{1 - \frac{p}{d}}.$$ Our statement  easily follows.

\medskip
Now, let us address the general case where the support is not bounded. Let $ \ell \geq 1$ and $Q = [-\ell, \ell]^d$. By Lemma~\ref{prop:part2},
$$
2 \diam(Q)^p \sqrt{  n \int f} + \E L( n f   ) \geq \E L_{\partial Q}( n f \ind_Q ). 
$$
Also, the above argument has shown that
$$
 \liminf_{n} \frac{ \E L_{\partial Q}( n f \ind_Q ) }{n ^{1 - \frac{p}{d}}}  \geq   \beta'_L \int_Q f^{1 - \frac{p}{d}} .
$$
We deduce that for any $Q = [-\ell, \ell]^d$,
$$
\liminf_{n} \frac{ \E L( n f  ) }{n ^{1 - \frac{p}{d}}} \geq \beta'_L \int_Q f^{1 - \frac{p}{d}} .
$$
Taking $\ell$ arbitrary large we obtain the claimed lower bound. 
\end{proof}

\subsubsection{Dealing with the singular component}\label{sec:singular}

In this section we explain how to extend Theorem~\ref{th:lower} from measures with densities
to general measures. Given a measure $\mu$, we consider its decomposition $\mu=\mu_{ac}+\mu_s$
into an absolutely continuous part and a singular part.

Our starting point is the following lemma, which can be viewed as an inverse subbadditivity property.
\begin{lemma}\label{le:mino-point}
Let $p\in (0,1]$ and $L=M_p$. Let $\cX_1,\cX_2,\cY_1,\cY_2$ be four finite  multisets included in a bounded 
set $Q$. Then 
$$L(\cX_1,\cY_1)\le L(\cX_1\cup\cX_2,\cY_1\cup\cY_2)+L(\cX_2,\cY_2)+\mathrm{diam}(Q)^p
 \Big(|\cX_1(Q)-\cY_1(Q)|+ 2 |\cX_2(Q)-\cY_2(Q)|\Big).$$
\end{lemma}
\begin{proof}
 Let us start with an optimal matching achieving  $L(\cX_1\cup\cX_2,\cY_1\cup\cY_2)$ and an optimal matching achieving $L(\cX_2,\cY_2)$. Let us view them as bipartite graphs $G_{1,2}$ and $G_2$ on the vertex sets $(\cX_1\cup\cX_2,\cY_1\cup\cY_2)$ and $(\cX_2,\cY_2)$ respectively (note that if a point appears more than
 once, we consider its instances as different graph vertices).
Our goal is to build a possibly suboptimal matching of $\cX_1$ and $\cY_1$. Assume without loss of generality that $\cX_1(Q)\le \cY_1(Q)$. Hence we need to build an injection from $\sigma:\cX_1\to \cY_1$ and to upper bound
its cost $\sum_{x\in\cX_1} |x-\sigma(x)|^p$.

  To do this, let us consider the graph $G$ obtained as the union of $G_{1,2}$ and $G_2$ (allowing multiple edges when
  two points are neighbours in both graphs).
  It is clear that in $G$ the points from $\cX_1$ and $\cY_1$ have degree at most one, while the points
  from $\cX_2$ and $\cY_2$ have degree at most 2.
  For each $x\in \cX_1$, let us consider its connected component $C(x)$ in $G$. Because of the above degree 
  considerations (and since no point is connected to itself in a bipartite graph) it is obvious that
  $C(x)$ is a path. 
  
   It could be that $C(x)=\{x\}$, in the case when $x$ is a leftover point in the matching corresponding 
    to  $G_{1,2}$. This means that $x$ is a point in excess and there are at most
    $|\cX_1(Q)+\cX_2(Q)-(\cY_1(Q)+\cY_2(Q))|$ of them.
   
\begin{figure}[htb]
 \begin{center}
\begin{psfrags}
\psfrag{x}{\Large $x$}
\psfrag{y}{\Large $y $}
  \includegraphics[angle=0,width = 5cm]{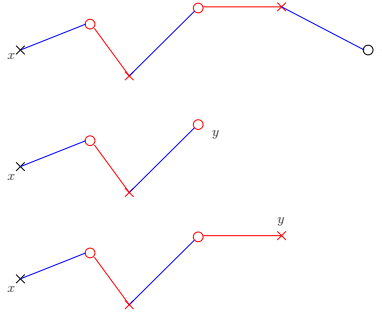}
 \end{psfrags} 
 \caption{\label{fig:matchcolor}The three possibilities for the path $C(x)$. In blue, $G_{1,2}$, in red $G_2$, the points in $\cX_1 \cup \cX_2$ are represented by a cross, points in $\cY_1 \cup \cY_2$ by a circle.}
\end{center}\end{figure}

   Consider now the remaining case, when  $C(x)$ is a non trivial path. Its first edge belongs to $G_{1,2}$. If there
   is a second edge, it has to be from $G_2$ (since $G_{1,2}$ as degree at most one). Repeating the argument,
   we see that the edges of the path are alternately from $G_{1,2}$ and from $G_2$. Note also 
   that the successive vertices are alternately from $\cX_1\cup\cX_2$ and from  $\cY_1\cup\cY_2$ (see Figure \ref{fig:matchcolor}). There are three possibilities:
\begin{itemize}
	\item The other end of the path is a point $y\in \cY_1$. In this case we are done, we have associated a point
	 $y\in \cY_1$ to our point $x\in \cX_1$. By the triangle inequality and since $(a+b)^p\le a^p +b^p$ due to the 
	 assumption $p\le 1$, $|x-y|^p$ is upper bounded by the sum of the $p$-th powers of the length of the edges
	 in $C(x)$.
	 \item The other end of the path is a point $y\in \cY_2$. The last edge is from $G_{1,2}$. So necessarily,
	 $y$ has no neighbour in $G_2$. This means that it is not matched. There are at most $|\cX_2(Q)-\cY_2(Q)|$ 
	  such points in the matching $G_2$.
	\item   The other end of the path is a point $x'\in \cX_2$. The last edge is from $G_{2}$. So necessarily,
	 $x'$ has no neighbour in $G_{1,2}$. This means that it is not matched in $G_{1,2}$. As already 
	 mentionend there are at most $|\cX_1(Q)+\cX_2(Q)-(\cY_1(Q)+\cY_2(Q))|$ 
	  such points.
\end{itemize}
Eventually we have found a way to match the points from $\cX_1$, apart maybe $|\cX_2(Q)-\cY_2(Q)|+|\cX_1(Q)+\cX_2(Q)-(\cY_1(Q)+\cY_2(Q))|$ of them. We match the latter points arbitrarily
to (unused) points in $\cY_1$ and upper bound the distances between matched points by $\mathrm{diam}(Q)$.
\end{proof}
As a direct consequence, we obtain:

\begin{lemma}\label{lem:is}
Let $\mu_1$ and $\mu_2$ be two finite measures supported in a bounded set $Q$.
Let $p\in (0,1]$ and $L=M_p$ be the bipartite matching functional. Then
$$ \E L(\mu_1)\le \E L(\mu_1+\mu_2) +\E L( \mu_2)+ 3 \, \diam(Q)^p \big(\sqrt{\mu_1(Q)}+\sqrt{\mu_2(Q)}\big).$$
\end{lemma}

\begin{proof}
Let $\cX_1,\cX_2,\cY_1,\cY_2$ be four independent Poisson point processes. Assume that for $i\in\{1,2\}$,
$\cX_i$ and $\cY_i$ have intensity measure $\mu_i$. Consequently $\cX_1\cup\cX_2$ and $\cY_1\cup \cY_2$
are independent Poisson point processes with intensity $\mu_1+\mu_2$. Applying the preceeding lemma~\ref{le:mino-point}
and taking expectations yields
$$ \E L(\mu_1)\le \E L(\mu_1+\mu_2)+\E L(\mu_2)+ 2\diam(Q)^p \big(\E |\cX_1(Q)-\cY_1(Q)|+  \E |\cX_2(Q)-\cY_2(Q)|\big).$$ 
As usual, we conclude using that
$$ \E |\cX_i(Q)-\cY_i(Q)|\le \sqrt{\E\big( (\cX_i(Q)-\cY_i(Q))^2\big)}=\sqrt{2\mathrm{var}(\cX_i(Q))}=\sqrt{2\mu_i(Q)}.$$
\end{proof}

\begin{theorem}\label{th:poisson-general}
Assume that $d\in\{1,2\}$ and $p\in(0,d/2)$, or that $d\ge 3$ and $p\in(0,1]$.
Let $L=M_p$ be the bipartite matching functional. Let $\mu$ be a finite measure on $\R^d$ with bounded support.
Let $f$ be the density of the absolutely continuous part of $\mu$. Assume that there exists $\alpha>\frac{2dp}{d-2p}$ such that $\int |x|^\alpha d\mu(x) <+\infty$. Then 
$$\liminf_n \frac{\E L(n\mu)}{n^{1-\frac{p}{d}}} \geq \beta'_L  \int_{\R^d} f^{1-\frac{p}{d}}.$$
Moreover if $f$ is proportional to the indicator function of a bounded set with positive Lebesgue measure 
$$
\lim_n \frac{\E L(n\mu)}{n^{1-\frac{p}{d}}} =  \beta_L  \int_{\R^d} f^{1-\frac{p}{d}}.
$$
\end{theorem}

\begin{proof}  Note that in any case, $p\le 1$ is assumed. Let us write $\mu=\mu_{ac}+\mu_s$ where $d\mu_{ac}(x)=f(x)dx$ is the absolutely continuous
part and $\mu_s$ is the singular part of $\mu$.

The argument is very simple if $\mu$ has a bounded support:
 apply the previous lemma with $\mu_1=n\mu_{ac}$ and $\mu_2=n\mu_s$. When $n$ tends to infinity, observing
that $\sqrt n$ is negligible with respect to $n^{1-\frac{p}{d}}$, we obtain that 
$$ \liminf_n \frac{\E L(n\mu_{ac})}{n^{1-\frac{p}{d}}}\le \liminf_n \frac{\E L(n\mu )}{n^{1-\frac{p}{d}}}+
 \limsup_n \frac{\E L(n\mu_{s})}{n^{1-\frac{p}{d}}}.$$
Observe that the latter upper limit is equal to zero thanks to  Theorem~\ref{th:up-poisson-general} applied to a purely singular measures. Eventually
$\liminf_n \frac{\E L(n\mu_{ac})}{n^{1-\frac{p}{d}}} \geq \beta'_L  \int_{\R^d} f^{1-\frac{p}{d}}$
by Theorem~\ref{th:lower} about absolutely continuous measures.

If $f$ is proportional to  an indicator function, we simply use scale invariance and Theorem \ref{th:lowerset} in place of Theorem~\ref{th:lower}.

Let us consider the general case of unbounded support. Let $Q=[-\ell,\ell]^d$ where $\ell>0$ is arbitrary. Let $\cX_1,\cY_1,\cX_2,\cY_2$ be four independent Poisson point processes, such that $\cX_1$ and $\cY_1$
have intensity measure $n \1_Q \cdot \mu_{ac}$, and  $\cX_2$ and $\cY_2$
have intensity measure $n (\mu_s+ 1_{Q^c}\cdot \mu_{ac})$. It follows that
$\cX_1\cup \cX_2$ and $\cY_1\cup\cY_2$ are independent Poisson point processes with intensity $n\mu$.
Set $T:= \max\{ |z|;\; z\in \cX_1\cup \cX_2\cup \cY_1\cup\cY_2\}$. Applying Lemma~\ref{le:mino-point} gives
$$ L(\cX_1,\cY_1) \le  L(\cX_1\cup\cX_2,\cY_1\cup\cY_2)+ L(\cX_2,\cY_2)+c_p T^p
 \big(|\mathrm{card}(\cX_1)-\mathrm{card}(\cY_1)|-|\mathrm{card}(\cX_2)-\mathrm{card}(\cY_2)| \big).$$
Taking expectations, applying the Cauchy-Schwarz inequality twice and Lemma~\ref{lem:ET} (note that
$\alpha>2p$) gives
\begin{eqnarray*}
   \E L(n f \1_Q)& \le& \E L(n\mu)+ \E L\big(n(\mu_s+\1_{Q^c}\mu_{ac})\big)+ c_p \sqrt{E\big[T^{2p}\big]} \left(
 \sqrt{2n \mu_{ac}(Q)}+\sqrt{2n (\mu_s(\R^d)+\mu_{ac}(Q^c))} \right)\\
  &\le &  \E L(n\mu)+ \E L\big(n(\mu_s+\1_{Q^c}\mu_{ac})\big)+ c'_p n^{\frac{p}{\alpha}+\frac{1}{2}}. 
\end{eqnarray*}
Since $\alpha>\frac{2dp}{d-2p}$ we obtain
$$ \liminf_n \frac{\E L(n\mu)}{n^{1-\frac{p}{d}}} \ge  \liminf_n \frac{\E L(n f \1_Q)}{n^{1-\frac{p}{d}}}
 -  \limsup_n \frac{\E L(n(\mu_s+ \1_Q^c\cdot \mu_{ac}))}{n^{1-\frac{p}{d}}}
  \ge \beta'_L  \int_Q f^{1-\frac{p}{d}}- \beta_L\int_{Q^c} f^{1-\frac{p}{d}},$$
where we have used Theorem~\ref{th:lower} for the lower limit  for bounded absolutely continuous measures
 and Theorem~\ref{th:up-poisson-general} for the upper limit. Recall that $Q=[-\ell,\ell]^d$. It remains
to let $\ell$ tend to infinity.
\end{proof}

Actually, using  classical duality techniques (which are specific to the bipartite matching) we can derive 
the following improvement of Lemma~\ref{lem:is}, which can be seen as an average monotonicity property:

\begin{lemma}\label{le:mino}
Let $p\in(0,1]$ and $L=M_p$. 
    Let $\mu_1$ and $\mu_2$ be two finite  measures supported on a bounded subset $Q\subset \R^d$. Then
  $$ \E L( \mu_1)  \le   \E L (\mu_1+\mu_2) +3 \mathrm{diam}(Q)^p \big( \sqrt{\mu_1(Q)}+ \sqrt{\mu_2(Q)} \big).$$
\end{lemma}
\begin{proof}
Since $p\in(0,1]$, the unit cost $c(x,y):=|x-y|^p$ is a distance on $\R^d$. The Kantorovich-Rubinstein
dual representation of the minimal matching cost (or optimal transportation cost) is particularly simple in this case (see e.g. \cite{rachev, villani,T}): for $\{x_1,\ldots,x_n\}$, $\{y_1,\ldots,y_n\}$ two multisets in $Q$,
$$ L\big(\{x_1,\ldots,x_n\} ,  \{y_1,\ldots,y_n\}\big)=\sup_{f\in \mathrm{Lip}_{1,0}} \sum_i f(x_i)-f(y_i),$$
where $\mathrm{Lip}_{1,0}$ denotes the set of function $f:Q\to \R$ which are 1-Lipschitzian for the distance
$c(x,y)$ (hence they are $p$-H\"olderian for the Euclidean distance) and vanish at a prescribed point $x_0\in Q$.
Observe that any function in  $\mathrm{Lip}_{1,0}$ is bounded by $\mathrm{diam}(Q)^p$ pointwise.

Let $\cX=\{X_1,\ldots,X_{N_1}\}$ and $\cY=\{Y_1,\ldots,Y_{N_2}\}$
be independent Poisson point processes with intensity $\mu$ of finite mass and supported on a set $Q$ 
of diameter $D<+\infty$.
 By definition, on the event $\{N_1\le N_2\}$,
 \begin{eqnarray*}
    L(\cX,\cY)&=& \inf_{A\subset \{1,\ldots, N_2\}; \mathrm{ card}(A)=N_1} L\big( \{X_i,1\le  i\le N_1\},
 \{Y_j,j\in A\}\big) \\
  &=&  \inf_{A\subset \{1,\ldots, N_2\}; \mathrm{ card}(A)=N_1} \sup_{f\in \mathrm{Lip}_{1,0}} 
     \left( \sum_{i\le N_1} f(X_i)-\sum_{j\in A} f(Y_j)\right) \\
  &\ge&  \sup_{f\in \mathrm{Lip}_{1,0}} 
     \left( \sum_{i\le N_1} f(X_i)-\sum_{j\le N_2} f(Y_j)\right) -D^p |N_1-N_2|
 \end{eqnarray*}
where we have used Kantorovich-Rubinstein duality to express the optimal matching of two samples of the same size
and  used that every $f\in \mathrm{Lip_{1,0}}$ satisfies $|f|\le D^p$  pointwise on $Q$.
A similar lower bound is valid when $N_1\ge N_2$. Hence, taking expectation and bounding
$E|N_1-N_2|$ in terms of the variance of the number of points in one process, one gets
\begin{equation}\label{eq:dual1} 
 \E L(\mu) \ge \E \sup_{f\in \mathrm{Lip}_{1,0}} \left( \sum_{i\le N_1} f(X_i)-\sum_{j\le N_2} f(Y_j)\right)
     -D^p \sqrt{2 |\mu|}.
\end{equation}
A similar argument also gives the following upper bound
\begin{equation}\label{eq:dual2}
 \E L(\mu) \le \E \sup_{f\in \mathrm{Lip}_{1,0}} \left( \sum_{i\le N_1} f(X_i)-\sum_{j\le N_2} f(Y_j)\right)
     +D^p \sqrt{2 |\mu|}.
\end{equation}

Let $\cX_1,\cX_2,\cY_1,\cY_2$ be four independent Poisson point processes. Assume that for $i \in \{1,2\}$,
 $\cX_i$ and  $\cY_i$ have intensity $ \mu_i$. As already mentioned, $\cX_1\cup \cX_2$ and $\cY_1\cup\cY_2$
are independent with common intensity $ \mu_1+\mu_2 $. Given a compact set $Q$ containing the 
supports of both measures, and $x_0\in Q$ we define the set $\mathrm{Lip}_{1,0}$.
Using \eqref{eq:dual1},
\begin{eqnarray*}
 \lefteqn{  \E L(\mu_1+\mu_2) =  \E L(\cX_1\cup \cX_2, \cY_1\cup\cY_2)} \\
  &\ge & \E \sup_{f\in \mathrm{Lip}_{1,0}} \left( \sum_{x_1\in \cX_1} f(x_1) - \sum_{y_1\in \cY_1} f(y_1)
    + \sum_{x_2\in \cX_2} f(x_2) - \sum_{y_2\in \cY_2} f(y_2) \right) -D^p \sqrt{2|\mu_1+\mu_2|}
\end{eqnarray*}
Now we use the easy inequality $\E \sup \ge \sup \E$ when $\E$ is
 the conditional expectation given $\cX_1,\cY_1$. Since $(\cX_2,\cY_2)$ are independent from 
 $(\cX_1,\cY_1)$, we obtain
 \begin{eqnarray*}
\lefteqn{ \E L(\mu_1+\mu_2)+D^p \sqrt{2|\mu_1+\mu_2|}}\\
  &\ge&  \E \sup_{f\in \mathrm{Lip}_{1,0}} \left( \sum_{x_1\in \cX_1} f(x_1) - \sum_{y_1\in \cY_1} f(y_1)
    +\E\Big( \sum_{x_2\in \cX_2} f(x_2) - \sum_{y_2\in \cY_2} f(y_2) \Big) \right)\\
  &=&  \E \sup_{f\in \mathrm{Lip}_{1,0}} \left( \sum_{x_1\in \cX_1} f(x_1) - \sum_{y_1\in \cY_1} f(y_1)
    \right)\\
  &\ge & \E L(\mu_1) - D^p \sqrt{2 |\mu_1|},
\end{eqnarray*}
where we have noted that the inner expectation vanishes and used \eqref{eq:dual2}.
The claim easily follows. \end{proof}

\subsection{Euclidean combinatorial optimization}

Our proof for the lower bound for matchings extends to some combinatorial optimization functionals $L$ defined by \eqref{eq:combopt}. In this paragraph, we explain how to adapt the above argument at the cost of ad-hoc assumptions on the collection of graphs $(\mathcal G_n)_{n \in \mathbb N}$. As motivating example, we will treat completely the case of the  bipartite traveling salesperson tour.

\subsubsection{Boundary functional}
\label{subsubsec:BFECO}

Let $S \subset \mathbb R^d$ and $\varepsilon , p\geq 0$. Set $q=2^{p-1}\wedge 1$. In what follows, $p$ is fixed and will be omitted in most places
where it would appear as an index. Given multisets
$\cX=\{X_1,\ldots,X_n\}$ and $\cY=\{Y_1,\ldots,Y_n\}$ included in $\mathbb R^d$, we first set
$$
   L^0_{\partial S,\varepsilon}(\cX , \cY)
   = \min_{G \in \mathcal G_n}\left\{ 
      \sum_{(i,j)\in [n]^2 : \{i , n+j \} \in G } d_{S,\varepsilon,p}(X_i,Y_{j})  
\right\}, 
$$
where 
\begin{equation}
  d_{S,\varepsilon,p}(x,y)  =\left\{ 
  \begin{array}{ccc}
     |x-y|^p & \mathrm{if}& x,y\in S,\\
     0  & \mathrm{if}& x,y\not\in S,\\
     q \big(\mathrm{dist}(x,S^c)^p+\varepsilon^ p\big) & \mathrm{if}& x\in S,\, y\not\in S\\
     q \big(\mathrm{dist}(y,S^c)^p+\varepsilon^ p\big) & \mathrm{if}& y\in S,\, x\not\in S
  \end{array}  
  \right.
\end{equation}
Now, if $\cX$ and $\cY$ are in $S$, we define the penalized boundary functional as
\begin{equation}
 L_{\partial S,\varepsilon}(\cX , \cY) = \min_{ A , B \subset S^c  }{L^0_{\partial S,\varepsilon}(\cX \cup A , \cY \cup B)}  , \label{eq:LLboundary}
\end{equation}
where the minimum is over all  multisets  $A$ and $B$  in $S^c$ such that $\mathrm{card}(\cX\cup A)    = \mathrm{card}(\cY\cup B) \geq \kappa_0$.
 When $\varepsilon=0$ we simply write $L_{\partial S}$. The main idea of this definition is to consider all possible configurations outside 
 the set $S$ but not to count the distances outside of $S$ (from a metric view point, all of $S^c$ is identified to a point which is at distance 
 $\varepsilon$ from $S$).
 
 The existence of the minimum  in \eqref{eq:LLboundary}  is due to the fact that $L^0_{\partial S}(\cX \cup A , \cY \cup B)$ can only take finitely many values less than any positive value (the quantities involved are just sums of distances between points of $\cX,\cY$ and of their distances
 to $S^c$). Notice that definition \eqref{eq:LLboundary} is consistent with the definition of the boundary functional for the matching functional $M_p$, given by \eqref{eq:Lboundary}. If $\cX$ and $\cY$ are
independent Poisson point processes with intensity $\nu$ supported in
$S$ and with finite total mass, we write $L_{\partial S,
  \varepsilon}(\nu)$ for the random variable $L_{\partial S,
  \varepsilon}(\cX,\cY)$.  
Also note that $d_{S,0,p}(x,y)\le |x-y|^p$. Consequently if $\card(\cX)=\card(\cY)$ then
\begin{equation}\label{eq:LLd}
L^0_{\partial S}(\cX,\cY)\le L(\cX,\cY).
\end{equation}

The next lemma will be used to reduce to uniform distributions  on squares.

\begin{lemma}\label{lem:Lcoupling2}
Assume (A1-A5). Let $\mu,\mu'$ be two probability measures on $\R^d$ with supports in $Q$ and $ n >0$. Then, for some constant $c$ depending only on $\kappa, \kappa_0$, 
$$   \E L_{\partial Q}(n\mu)\le  \E L_{\partial Q}(n\mu') + 2 c n \,\mathrm{diam}(Q)^p\, d_\mathrm{TV}(\mu,\mu').$$
Consequently, if $f$ is a non-negative locally integrable function on
   $\R^d$, setting $\alpha=\int_Q f/\mathrm{vol}(Q)$, it holds
  $$ \E L_{\partial Q}(nf \1_Q) \le \E L_{\partial Q}(n\alpha \1_Q) + c n\, \mathrm{diam}(Q)^p \, \int_Q|f(x)-\alpha|\, dx.$$
  \end{lemma}

\begin{proof}
The functional $ L_{\partial Q}$ satisfies a slight modification of property $(\mathcal R_p)$ :  for all multisets $\mathcal X,\mathcal Y,
\mathcal X_1,\mathcal Y_1,\mathcal X_2,\mathcal Y_2$ in $Q$, it holds 
\begin{equation}
L_{\partial Q}(\mathcal X\cup \cX_1,\cY \cup\cY_1)\le L_{\partial Q}(\mathcal X\cup \cX_2,\cY \cup\cY_2) + C \mathrm{diam}(Q) ^p \big( \mathrm{card}(\cX_1)+
\mathrm{card}(\cX_2)+\mathrm{card}(\cY_1)+\mathrm{card}(\cY_2)\big), \label{eq:regLbound}
\end{equation}  
with $C = C(\kappa , \kappa_0 )$. The above inequality is established as in the proof of Lemma \ref{le:Rp}. Indeed, by linearity and symmetry we should check \eqref{eq:RpsansY1XY2} and \eqref{eq:RpsansY2XY1} for $L_{\partial Q}$. To prove \eqref{eq:RpsansY1XY2}, we consider an optimal triplet $(G,A,B)$ for $(\cX, \cY)$ and apply the merging property (A4) to $G$ with the empty graph and $m=1$ : we obtain a graph $G''$ and get a triplet $(G'',A,B \cup \{b\})$ for $(\cX \cup \{a\}, \cY)$, where $b$ is any point in $\partial Q$. To prove \eqref{eq:RpsansY2XY1}, we now consider an optimal triplet $(G,A,B)$ for $(\cX \cup \{a\}, \cY)$ and move the point $a$ to the $a'$ in $\partial Q$ in order to obtain a triplet $(G,A \cup\{a'\},B)$ for $(\cX, \cY)$. 

With \eqref{eq:regLbound} at hand, the statements  follow from the proofs of Proposition \ref{prop:approx} and Corollary \ref{cor:couplingaverage}. 
   \end{proof}

The next lemma gives a lower bound on $L$ in terms of its boundary functional and states an important superadditive property of $L_{\partial S}$. 
     
     \begin{lemma}\label{prop:Lpart2}
Assume (A1-A5). Let $\nu$ be a finite measure on $\mathbb R^d$ and consider a
        partition $Q=\cup_{P\in\mathcal P} P$ of a bounded subset of $\mathbb
        R^d$. Then, if $c = 4 \kappa ( 1 + \kappa_0)$, we have
       \begin{eqnarray*}
        c     \sqrt{\nu(\mathbb R^d)} \,  \mathrm{diam}(Q)^p  +  \E L( \nu)  \ge   \E L_{\partial Q} (\1_Q \cdot \nu)   \geq     \sum_{P\in\mathcal P} \E L_{\partial P} (\1_P\cdot \nu).
       \end{eqnarray*}
    \end{lemma}
    \begin{proof}
We start with the first inequality.  Let $\mathcal X=\{X_1,\ldots,X_m\},\mathcal
       Y=\{Y_1,\ldots,Y_n\}$ be multisets included in $Q$ and $\cX'=\{X_{m+1},\ldots,X_{m+m'}\}$, $\cY'=\{Y_{n+1},\ldots,Y_{n+n'}\}$ be multisets included in $Q^c$. First, let us show that 
     \begin{equation}\label{eq:LRbQ}
         c_1   |m+m'-n - n'| \mathrm{diam}(Q)^p + L(\cX\cup \cX',\cY\cup \cY') \geq L_{\partial Q}  (\cX,\cY), 
         \end{equation}  
      with $c_1 =  \kappa ( 1 + \kappa_0)$. To do so, let us consider  an optimal graph $G$ for $L(\cX\cup \cX',\cY\cup \cY')$.
      It uses all the points but $|m+m'-n-n'|$ points in excess.
       We consider the subsets $\cX_0 \subset \cX$ and $\cY_0 \subset \cY$ of points that are used in $G$ and belong to $Q$.
     By definition there exist subsets $A,B\subset Q^ c$ such that $\card(\cX_0\cup A)=\card(\cY_0\cup B)$      and
      $L(\cX\cup \cX',\cY\cup \cY')=L(\cX_0\cup A,\cY_0\cup B)$.
      By definition of the boundary functional and using \eqref{eq:LLd},
      $$ L_{\partial Q}(\cX_0,\cY_0)\le L^0_{\partial Q}(\cX_0\cup A,\cY_0\cup B) \le L (\cX_0\cup A,\cY_0\cup B)= L(\cX\cup \cX',\cY\cup \cY').$$
Finally, since there are at most  $|n +n' - m - m'|$ points in $\cX\cup \cY$ which are not in $\cX_0\cup \cY_0$ (i.e. points of $Q$ not used for the optimal
$G$), the modified \eqref{eq:Rp} property 
 given by Equation \eqref{eq:regLbound} yields \eqref{eq:LRbQ}.
 We apply the latter inequality to $\cX$, $\cY$ independent Poisson processes of intensity $\ind_Q \cdot \nu$, and $\cX'$, $\cY'$, two independent Poisson processes of intensity $\ind_{Q^c} \cdot \nu$, independent of $(\cX,\cY)$. Then $\cX \cup \cX'$, $\cY \cup \cY'$ are independent Poisson processes of intensity $\nu$. Taking expectation, we obtain the first inequality, with $c = 4 c_1$.

We now prove the second inequality. As above, let $\mathcal X=\{X_1,\ldots,X_m\},\mathcal
       Y=\{Y_1,\ldots,Y_n\}$ be multisets included in $Q$.  Let $G \in \mathcal G_k$ be an optimal graph for $L_{\partial Q}( \cX, \cY)$ and  $A = \{X_{m+1}, \cdots , X_{k}\}$, $B = \{Y_{n+1}, \cdots , Y_{k}\}$ be optimal sets in  $ Q^ c$.
Given this graph $G$ and a set $S$, we denote by $E^0_S$ the set of edges $\{i,k+j\}$ of $G$ such that $X_i \in S$ and $Y_j \in S$, by $E^1_S$ the set of edges $\{i,k+j\}$ of $G$ such that $X_i \in S$ and $Y_j \in S^c$, and by $E^2_S$ the set of edges $\{i,k+j\}$ of $G$ such that $X_i \in S^c$ and $Y_j \in S$. Then by definition of the boundary functional       
       \begin{eqnarray*}
L_{\partial Q}  ( \cX , \cY ) & =&  L^0_{\partial Q}  ( \cX \cup A , \cY \cup B) \\
&=& \sum_{\{ i , k+j\}  \in E^0_Q} |X_i-Y_{j}| ^p +     \sum_{ \{ i , k+j\} \in E^1_Q} q \, d(X_i,Q^c)^p +
               \sum_{ \{ i , k+j\} \in E^2_Q} q\, d(Y_j,Q^c) ^p .
      \end{eqnarray*}  
Next, we bound these sums from below by considering the cells of the partition $\mathcal P$. 
If $x \in Q$, we denote by $P(x)$ the unique $P \in \mathcal P$ that contains $x$. 

If an edge  $e=\{i , k+j \} \in G $   is such that
$X_i,Y_j$ belong to the same cell $P$, we observe that $e\in E^0 _P$ and we leave the quantity $|X_i-Y_j|^p$ unchanged.       

If on the contrary, $X_i$ and $Y_j$ belong to different cells,
from H\"older inequality, 
       $$ |X_i-Y_{j}|^p \ge q \, d(X_i, P(X_i)^c)^p  + q \, d(Y_j, P(Y_{j})^c)^p .$$

Eventually, for any boundary edge in $E^1_Q$, we lower bound the contribution $d(X_i,Q^c)^p$ by $d(X_i,P(X_i)^c)^p$ and we do the same for 
 $E^2_Q$. 
       Combining  these inequalities and grouping the terms according to the
         cell $P\in\mathcal P$ to which the points belong,
             \begin{eqnarray*}
   L_{\partial Q}  ( \cX , \cY ) &\ge &  \sum_{P\in\mathcal P}\left( \sum_{\{ i , k+j\}  \in E^0_P} |X_i-Y_{j}| ^p +     \sum_{ \{ i , k+j\} \in E^1_P} q \, d(X_i,\partial P)^p +
               \sum_{ \{ i , k+j\} \in E^2_P} q\, d(Y_j,\partial P) ^p \right).
         \end{eqnarray*}
For a given cell $P$, set $A' = ( \cX  \cup A) \cap P^c$ and $B' =  ( \cY  \cup B) \cap P^c$. We get 
\begin{eqnarray*}
&&   \sum_{\{ i , k+j\}  \in E^0_P} |X_i-Y_{j}| ^p +            \sum_{ \{ i , k+j\} \in E^1_P} q \, d(X_i,\partial P)^p +
               \sum_{ \{ i , k+j\} \in E^2_P} q\, d(Y_j,\partial P) ^p \\
                 &   =  &  L^0_{P^c} (  ( \cX \cap P ) \cup A'  , ( \cY \cap P) \cup B' )  
                \geq  L_{\partial P} ( \cX \cap P , \cY \cap P ). \end{eqnarray*}
So applying these inequalities to $\cX$ and $\cY$ two independent Poisson point processes with intensity $\nu \ind_Q$ and taking expectation, we obtain the claim.
      \end{proof}

Let $Q=[0,1]^d$ and denote $$\bar L_{\partial Q}(n)=\E L_{\partial Q}(n\1_Q).$$
\begin{lemma}
   \label{lem:LliminfBM}
Assume (A1-A5). Let $Q\subset \R^d$ be a cube of side-length 1.
 If $0< 2p < d$, then
$$ \lim_{n\to \infty} \frac{\bar L_{\partial Q}(n)}{n^{1-\frac{p}{d}}} =  \beta'_L,$$
where $\beta'_L >0$ is a constant depending on $L$, $p$ and $d$.
\end{lemma}

\begin{proof}
The proof is the same than the proof of Lemma \ref{lem:liminfBM}, with Lemma \ref{prop:Lpart2} replacing Lemma \ref{prop:part2}. 
\end{proof}

\subsubsection{General absolutely continuous measures with unbounded support}

\begin{theorem} \label{th:Llower} Assume (A1-A5) and that $0 < 2 p < d$. Let $f: \R^d \to \R^+$
   be an integrable function.  Then $$ \liminf_{n}
   \frac{\E L(n f )}{n^{1 - \frac{p}{d}}} \geq  \beta'_L \int_{\R^d}
   f^{1 - \frac{p}{d}}.$$
\end{theorem}

\begin{proof}
 The proof is now formally the same than the proof of Theorem \ref{th:lower}, invoking Lemmas \ref{lem:Lcoupling2}, \ref{prop:Lpart2} and \ref{lem:LliminfBM} in place of Lemmas  \ref{lem:coupling2}, \ref{prop:part2} and \ref{lem:liminfBM} respectively. 
\end{proof}

\begin{remark} Finding good lower bounds for a general bipartite functional $L$  on $\mathbb R^d$ satisfying the properties $(\mathcal H_p)$, $(\mathcal R_p)$, $(\mathcal S_p)$ could be significantly more difficult. It is far from obvious to define a proper boundary functional $L_{\partial Q}$ at this level of generality. However if there exists a bipartite functional $L_{\partial Q}$ on $\mathbb R^d$ indexed on sets $Q \subset \R^d$ such that for any $t >0$, $\E L_{\partial (t Q) } (n \ind_{t Q}) = t^{p} \E L_{\partial Q} (n t^d \ind_{Q})$ and such that the statements of Lemmas \ref{prop:Lpart2}, \ref{lem:Lcoupling2}, \ref{lem:LliminfBM} hold, then the statement of Theorem \ref{th:lower} also holds for the functional $L$. Thus, the caveat of this kind of techniques lies in the good definition of a boundary functional $L_{\partial Q}$. 
\end{remark}

\subsubsection{Dealing with the  singular component. Example of the traveling salesperson problem.}
Let $p\in(0,1]$. We shall say that a bipartite functional $L$ on $\R^d$ satisfies the inverse subadditivity property
$(\mathcal I_p)$ if there is a constant $C$ such that for all finite multisets $\cX_1,\cY_1,\cX_2,\cY_2$ included in a bounded set $Q\subset \R^d$,
$$ L(\cX_1,\cY_1) \le L(\cX_1\cup\cX_2,\cY_1 \cup\cY_2)+ L(\cX_2,\cY_2)  + C \mathrm{diam}(Q)^p \big(1+ |\cX_1(Q)-\cY_1(Q)|
  + |\cX_2(Q)-\cY_2(Q)|\big).$$
Although it makes sense for all $p$, we have been able to check this property on examples only for 
$p\in(0,1]$. Also we could have added a constant in front of $L(\cX_2,\cY_2)$.

It is plain that the argument of Section~\ref{sec:singular} readily adapts to a functional satisfying
$(\mathcal I_p)$, for which one already knows a general upper limit result and a limit result for 
absolutely continuous laws. It therefore provides a limit result for general laws. In the remainder of this section, we show that the  traveling salesperson bipartite tour functional $L=T_p$, 
$p \in(0, 1]$  enjoys the inverse subadditivity property.
This allows to prove the following result:

\begin{theorem} \label{th:bTSP} Assume that either $d\in \{1,2\}$ and $0 < 2 p < d$, or $d\ge 3$ and $p\in (0,1]$.
Let $L = T_p$ be the traveling salesperson bipartite tour functional. Let $\mu$ be a finite measure such that  for some $\alpha >   \frac{2dp}{ d -2p}$,
$
\int |x|^{\alpha} d \mu  < +\infty.
$
  Then, if $f$ is a density function for the absolutely continuous part of $\mu$, 
   $$ \liminf_{n}
   \frac{\E L(n \mu  )}{n^{1 - \frac{p}{d}}} \geq \beta'_L \int_{\R^d}
   f^{1 - \frac{p}{d}}.$$
   Moreover if $f$ is proportional to the indicator function of a bounded set with positive Lebesgue measure 
$$
\lim_n \frac{\E L(n\mu)}{n^{1-\frac{p}{d}}} =  \beta_L  \int_{\R^d} f^{1-\frac{p}{d}}.
$$
\end{theorem}

All we have to do is to check property $(\mathcal I_p)$. More precisely:
\begin{lemma}\label{le:euler}
Assume $p\in(0,1]$ and $L = T_p$. For any set $\cX_1, \cX_2, \cY_1, \cY_2$ in a bounded set $Q$ 
\begin{eqnarray*}
L ( \cX_1 , \cY_1 ) & \leq & L ( \cX_1 \cup \cX_2 , \cY_1 \cup \cY_2 ) + L ( \cX_2, \cY_2 ) \\ 
& & \quad + \; 2 \, \mathrm{diam}(Q)^p\left( 1 + | \mathrm{\card}(\cX_1)- \mathrm{\card}(\cY_1)|+ | \mathrm{\card}(\cX_2)   - \mathrm{\card}(\cY_2)|\right). 
\end{eqnarray*}
\end{lemma}

\begin{proof}
We may assume without loss of generality that $\mathrm{\card} ( \cX_1) \wedge \mathrm{\card}(\cY_1) \geq 2$, otherwise, $L ( \cX_1, \cY_1 ) = 0$ and there is nothing to prove.  Consider an optimal cycle $G_{1,2}$ for  $L(\cX_1 \cup \cX_2 , \cY_1 \cup \cY_2)$. In $G_{1,2}$, $m = | \mathrm{\card}(\cX_1)+ \mathrm{\card}(\cY_1) -  \mathrm{\card}(\cX_2)   - \mathrm{\card}(\cY_2)| \leq | \mathrm{\card}(\cX_1)- \mathrm{\card}(\cY_1)|+ | \mathrm{\card}(\cX_2)   - \mathrm{\card}(\cY_2)|$ points have been left aside. We shall build a bipartite tour $G_1$ on $(\cX'_1, \cY'_1)$, the points of $(\cX_1, \cY_1)$ that have not been left aside by $G_{1,2}$. 

\begin{figure}[htb]
 \begin{center}
  \includegraphics[angle=0,width = 6cm]{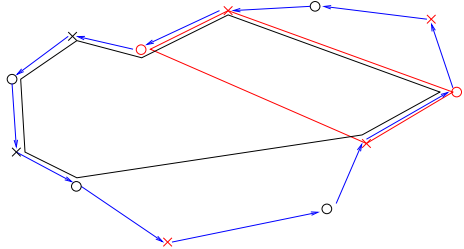}
 \caption{\label{fig:matchcolor} In blue, the oriented cycle $G_{1,2}$, in red $G_2$, in black $G'_{1,2}$. The points in $\cX_1 \cup \cX_2$ are represented by a cross, points in $\cY_1 \cup \cY_2$ by a circle.}
\end{center}\end{figure}

We consider an optimal cycle $G_2$ for $L( \cX'_2 ,  \cY'_2)$, where $(\cX'_2, \cY'_2)$ are the points of $(\cX_2, \cY_2)$ that have not been left aside by $G_{1,2}$. We define $(\cX''_2, \cY''_2) \subset (\cX'_2, \cY'_2)$ as the sets of points that are in $G_2$. Since $\mathrm{\card} ( \cX'_1) + \mathrm{\card} ( \cX'_2) = \mathrm{\card} ( \cY'_1) + \mathrm{\card} ( \cY'_2)$, we get $\mathrm{\card} ( \cX'_1) - \mathrm{\card} ( \cY'_1)  =  -  \mathrm{\card} ( \cX'_2) + \mathrm{\card} ( \cY'_2)$. It implies that the same number of points from the opposite type need to be removed in $( \cX'_1,  \cY'_1)$ and $( \cX'_2 ,  \cY'_2)$ in order to build a bipartite tour. We fix an orientation on $G_{1,2}$. Assume for example that $\mathrm{\card} ( \cX'_2) \geq \mathrm{\card} ( \cY'_2)$, if a point $x \in \cX'_2 \backslash \cX''_2$, we then remove the next point $y$ on the oriented cycle $G_{1,2}$ of $\cY'_1$. Doing so, this defines a couple of sets $(\cX''_1, \cY''_1) \subset (\cX'_1, \cY'_1)$ of cardinality $\mathrm{\card} ( \cX'_1) \wedge \mathrm{\card}(\cY'_1)$ and
$$
L  (\cX'_1, \cY'_1)  \leq L ( \cX''_1, \cY''_1).
$$
 We define $G'_{1,2}$ as the cycle on $(\cX''_1 \cup \cX''_2, \cY''_1\cup \cY''_2)$ obtained from $G_{1,2}$ by saying that the point after $x \in(\cX''_1 \cup \cX''_2, \cY''_1\cup \cY''_2)$ in the oriented cycle $G'_{1,2}$ is the next point $y \in(\cX''_1 \cup \cX''_2, \cY''_1\cup \cY''_2)$ in $G_{1,2}$. By construction, $G'_{1,2}$ is a bipartite cycle. Also, since $p \in (0, 1]$, we may use the triangle inequality : the distance between two successive points in the circuit $G'_{1,2}$ is bounded by the sum of the length of the intermediary edges in $G_{1,2}$. We get 
$$
L (\cX''_1 \cup \cX''_2, \cY''_1\cup \cY''_2) \leq L(\cX_1 \cup \cX_2 , \cY_1 \cup \cY_2).
$$

Now consider the (multi) graph $G = G'_{1,2} \cup G_2$ obtained by adding all edges of $G'_{1,2}$ and $G_2$. This graph is bipartite, connected, and points in $(\cX''_1, \cY''_1)$ have degree $2$  while those  in $(\cX''_2, \cY''_2)$ have degree $4$. Let $k$ be the number of edges in $G$, we recall that an eulerian circuit in $G$ is a sequence $E = (e_1, \cdots, e_{k})$ of adjacent edges in $G$ such that $e_{k}$ is also adjacent to $e_1$ and all edges of $G$ appears exactly once in the sequence $E$. By the Euler's circuit theorem, there exists an eulerian circuit in $G$. Moreover, this eulerian circuit can be chosen so that if $e_i = \{u_{i-1}, u_i\} \in G_2$  then $e_{i+1} = \{u_{i+1} , u_i\} \in G'_{1,2}$ with the convention that $e_{k+1} =e_1$.

This sequence $E$ defines an oriented circuit of points. Now we define an oriented circuit on $(\cX''_1, \cY''_1)$, by connecting a point $x$ of $(\cX''_1, \cY''_1)$ to the next point $y$ in $(\cX''_1, \cY''_1)$ visited by the oriented circuit $E$. Due to the property that  $e_i \in G_2$ implies $e_{i+1} \in G$, if $x \in \cX''_1$ then $y \in \cY''_1$ and conversely, if $x \in \cY''_1$ then $y \in \cX''_1$. Hence, this oriented circuit defines a bipartite cycle $G_1$ in $(\cX''_1, \cY''_1)$. 

By the triangle inequality, the distance between two successive points in the circuit $G_1$ is bounded by the sum of the length of the intermediary edges in $E$. Since each edge of $G$ appears exactly once in $E$, it follows that 
$$
L( \cX'_1, \cY'_1) \leq L(\cX_1 \cup \cX_2 , \cY_1 \cup \cY_2) + L(\cX'_2 , \cY'_2).
$$
To conclude, we merge  arbitrarily  to the cycle $G_1$ the remaining points of $(\cX_1,\cY_1)$, there are at most $m$ of them (regularity $(\mathcal R_p)$ property). \end{proof}

\section{Variants and final comments}\label{sec:final}

As a conclusion, we  briefly discuss variants and possible extensions of Theorem~\ref{th:mainMp}. For $d > 2p$ and when $\mu$ is
the uniform distribution on the cube $[0,1]^d$, there exists a
constant $\beta_p(d) >0$ such that almost surely
 $$ \lim_{n\to \infty } n^{\frac{p}{d}-1}  M_p \big(\{X_1,\ldots,X_n\}, \{Y_1,\ldots,Y_n\}\big)=\beta_p(d).$$
A natural question is to understand what happens below the critical line $d = 2p$, i.e. when $d \leq 2p$. For example for $d =2$ and $p =1$, a similar convergence is also expected in dimension 2 with scaling $\sqrt{n \ln n}$, but this is a
 difficult open problem. The main result in this direction goes back
 to Ajtai, Koml\'os and Tusn\'ady \cite{AKT}. See also the improved upper bound of Talagrand and Yukich in \cite{TY}. In dimension 1, there is
 no such stabilization to a constant.

Recall that $$\left(\frac 1 n  M_p \big( \{X_i \}_{i = 1}^n , \{Y_i \}_{i = 1}^n \big) \right)^{\frac 1 p} = W_p \left( \frac 1 n \sum_{i =1}^n \delta_{X_i} ,  \frac 1 n \sum_{i =1}^n \delta_{Y_i} \right).$$ 
where $W_p$ is the $L_p$-Wasserstein distance. A variant of Theorem~\ref{th:mainup} can be obtained along the same lines,
concerning the convergence of
 $$ n^{\frac{1}{d}} W_p \left( \frac{ 1 } {n }\sum_{i=1}^{n} \delta_{X_i},\, \mu\right),$$
where $\mu$ is the common distribution of the $X_i$'s. Such results are of fundamental 
importance in statistics. Also note that combining the triangle inequality and Jensen 
inequality, it is not hard to see that
$$ \E  W_1 \Big( \frac1n \sum_{i=1}^{n} \delta_{X_i},\,\mu  \Big)
\leq  \E W_1\Big( \frac1n \sum_{i=1}^{n} \delta_{X_i},\,\frac 1 n \sum_{i =1}^n \delta_{Y_i} \Big)
 \le 2 \E W_1\Big( \frac1n \sum_{i=1}^{n} \delta_{X_i},\, \mu \Big),$$
(similar inequalities hold for $p \geq 1$). Hence it is clear
that the behaviour of this functional is quite close to the one
of the two-sample optimal matching. However, the extension of Theorem~\ref{th:mainMp} would require some care in the definition of the boundary functional.

Finally, it is worthy to note that the case of uniform distribution for $L = M_p$ has a connection with stationary matchings of two independent Poisson point processes of intensity $1$, see Holroyd, Pemantle, Peres and Schramm  \cite{HPPS}. Indeed,  consider mutually independent random variables $(X_i)_{i\ge 1}$ and $(Y_j)_{j\ge 1}$ having uniform distribution on $Q = [-1/2,1/2]^d$. It is well known that for any $x$ in the interior of $Q$, the pair of point processes 
$$\left( \frac 1 n  \sum_{i = 1} ^n  \delta_{n^{\frac 1 d} ( X_i - x )  } , \frac 1 n  \sum_{i = 1} ^n  \delta_{n^{\frac 1 d} ( Y_i - x)  } \right)$$ converges weakly for the topology of vague convergence to 
$
(\Xi_1 , \Xi_2 ) 
$, where $\Xi_1$ and $\Xi_2$ are two independent Poisson point processes of intensity $1$. Also, we may write 
$$
n^{\frac p d - 1} \E M_p ( \{X_i \}_{i = 1} ^n , \{Y_i \}_{i = 1} ^n ) =  \frac 1 n \E \sum_{i = 1} ^n   \left| n^{\frac 1 d} ( X_i  - x) - n^{\frac 1 d} ( Y_{\sigma_n ( i )}- x)  \right| ^p.
$$
where $ \sigma_n$ is an optimal matching. Now, the fact that for $0 < p < 2d$, $\lim_n n^{\frac p d - 1} \E M_p ( \{X_i \}_{i = 1} ^n , \{Y_i \}_{i = 1} ^n ) = \beta_p ( d)$ implies the tightness of the sequence of matchings $\sigma_n$ and it can be used to define a stationary matching $\sigma$ on $(\Xi_1 , \Xi_2 )$, see the proof of Theorem 1 (iii)  in \cite{HPPS} for the details of such an argument. In particular, this matching $\sigma$ will enjoy a local notion of minimality for the $L_p$-norm, as defined by Holroyd in \cite{H} (for the $L_1$-norm). 
See also related work of Huesmann and Sturm \cite{Hu}. 

\section*{Acknowledgments}

We are indebted to Martin Huesmann for pointing an error in the proof of a previous version of Theorem \ref{th:mainMp}. This is also a pleasure to thank for its hospitality the Newton Institute where part of this work has been done (2011 Discrete Analysis programme).

\bigskip
\noindent
Address: Institut de Math\'ematiques (CNRS UMR 5219). Universit\'e
Paul Sabatier. 31062 Toulouse cedex 09. FRANCE.

\noindent
Email: barthe@math.univ-toulouse.fr, bordenave@math.univ-toulouse.fr

\end{document}